\documentclass[11pt,reqno]{amsart}
\usepackage{amsmath,mathtools,amsthm,amssymb}
\usepackage{graphicx}
\usepackage{epsfig,epsf,psfrag}
\usepackage{amsfonts}
\usepackage{setspace}
\usepackage{color}
\usepackage[font=small]{caption}
\usepackage[font=footnotesize]{subcaption}
\usepackage[top=1in, bottom=1in, left=1.2in, right=1.2in]{geometry}
\usepackage{afterpage}

\usepackage{isomath}

\newtheorem{theorem}{Theorem}[section]

\newtheorem{definition}[theorem]{Definition}
\newtheorem{lemma}[theorem]{Lemma}

\newtheorem{remark}[theorem]{Remark}

\title{A uniqueness theory on determining the nonlinear energy potential in phase-field system}

\author{Tianhao Ni\textsuperscript{1} and Jun Lai\textsuperscript{1,2}}
\thanks{\textsuperscript{1}School of Mathematical Sciences, Zhejiang University, Hangzhou, Zhejiang, 310027, China. \textsuperscript{2}Institute of Fundamental and Transdisciplinary Research, Zhejiang University, Hangzhou, Zhejiang, 310027, China.\\
Emails: \texttt{thni@zju.edu.cn}, \texttt{laijun6@zju.edu.cn}}

\subjclass[2020]{47J07, 74H75, 74A50, 74B20}

\keywords{Nonlinear inverse problems, Phase-field system, Cahn-Hilliard equations, Allen-Cahn equations, Higher order linearization}

\begin{document}
	
	\maketitle
	\begin{abstract} The phase-field system is a nonlinear model that has significant applications in material sciences. In this paper, we are concerned with the uniqueness of determining the nonlinear energy potential in a phase-field system consisting of Cahn–Hilliard and Allen–Cahn equations. This system finds widespread applications in the development of alloys engineered to withstand extreme temperatures and pressures. The goal is to reconstruct the nonlinear energy potential through the measurements of concentration fields. We establish the local well-posedness of the phase-field system based on the implicit function theorem in Banach spaces. Both of the uniqueness results for recovering time-independent and time-dependent energy potential functions are provided through the higher order linearization technique.
 
\end{abstract}
\section{Introduction}
  Grounded in thermodynamic theory, the phase-field method utilizes differential equations, mainly  the Cahn-Hilliard(CH) equations \cite{CH58} and  Allen-Cahn(AC) equations \cite{AC79}, to describe the dynamic evolution process of a system that is driven by diffusion, ordering potential, and thermodynamic driving forces~\cite{GLF10,LZT10}.  Due to the capability to handle arbitrary and complex morphological evolution  and achieve multi-field coupling without explicitly tracking interfaces, the phase-field method has a vast number of applications in the field of material sciences.  In this paper, we focus on the inverse problems associated with the phase-field system that couples the CH and AC equations. 
  
  This coupled system has found extensive applications in the design of alloys, which play a crucial role in areas requiring resistance to high temperature, complex stress, and surface stability \cite{R08,YLW18,TZN14}. To predict the long-term behavior of these alloys under extreme environments, it is essential to understand the creep mechanism.  While traditional experimental approaches for exploring creep mechanisms are both expensive and time-consuming, the phase-field method has emerged as a promising alternative by offering a rapid and cost-effective means of investigating alloy creep behavior~\cite{WK95b,RK99,K83}. However, recovering the elastic energy in the phase-field model poses a significant challenge, primarily due to the model's inherent nonlinearity as a function of the concentration field. In order to achieve precise and efficient simulations of alloy creep behavior, it becomes crucial to explore whether inferring elastic properties or the elastic potential function is feasible by observing the concentration field distribution within the material. If the potential function can be reconstructed through such measurements, one can utilize the phase-field model to simulate the precipitation process of different phases, as well as the microstructure evolution and deformation during creep, which provides important insights for  more reliable designs and optimizations of alloys. 

Given the significance of the phase-field model, extensive research has been conducted on the forward problems in the literature. The existence and uniqueness of weak solutions for the CH equation and AC equation under Dirichlet or Neumann boundary conditions were established in \cite{F89} and \cite{L10}, respectively. Moreover, the existence of non-negative solutions for the CH equation has been investigated in \cite{Y92}. The well-posedness of weak solutions for the coupled systems with logarithmic nonlinearities were also investigated, including the system coupling multiple CH equations~\cite{BJW01} and the system coupling CH and AC equations~\cite{BBG01}. The uniqueness of weak solutions for some special phase-field systems has been addressed in \cite{BDBE20,F06}. However, less attention has been paid to the local well-posedness of fourth-order semilinear systems with periodic boundary conditions, which have been widely used in the numerical simulations \cite{FTY15,CS20,CLW20}, although there have been some findings regarding the local well-posedness of semilinear parabolic systems \cite{LLL22,LZ23}.

Compared to advancements in the forward problems, much less progress has been made in addressing the inverse problems within phase-field systems and related nonlinear equations.  
In \cite{V93}, the first order linearization technique is employed to prove the uniqueness of inverting the nonlinear components of parabolic equations. In \cite{LLL22}, a higher order linearization approach is proposed to simultaneously invert both the nonlinear component and the initial value. In \cite{BEH23}, the uniqueness of determining energy function has been demonstrated by employing the weak formulation of CH equation. In \cite{KL20}, the inverse problem for the coupled Cahn-Hilliard-Chemotaxis system has been investigated from the perspective of optimal control theory. However, these efforts primarily concentrate on the inverse problems of determining the constant coefficients of nonlinear components. In this paper, we focus on the inverse problems associated with a phase-field system coupling the  CH and AC equations, where the unknown energy potential is  highly nonlinear. The experimental measurements often provide information solely about the concentration field.  Under such circumstances, the uniqueness of inverting the nonlinear potential independent of the space and time has been established in \cite{BEH23}. However, studies regarding uniqueness in more general settings, particularly those involving energy with variable elastic potentials \cite{LZT10} remain absent. Hence, determining the spatially and temporally varying nonlinear potential of the phase-field system based  on the concentration field measurements is both imperative and challenging. 

From this perspective, we consider two inverse problems:
\begin{itemize}
    \item[(IP1)] Determining the time-independent nonlinear energy based on the concentration measurement within the material at a single time point.
    \item[(IP2)] Determining the time-dependent nonlinear energy  based on the concentration measurements  within the material at multiple time points.
\end{itemize} 
A more precise description of these two problems will be given in the next section. We first establish the local well-posedness of the CH-AC system  under periodic boundary conditions, by relying on the implicit function theorem in Banach spaces and the existence of weak solutions for CH equations.  Subsequently, the uniqueness results for both inverse problems are studied  by employing the higher order linearization technique, which was first introduced in \cite{KLU14} to study the inverse problems in the Einstein scalar field equations. This technique  has also been successfully applied to the system of mean field games to establish the uniqueness of nonlinear coefficients and source terms \cite{LZ23}.

The rest of the paper is organized as follows. Section 2 provides the mathematical formulation and the physical background for the phase-field system. The main results are also presented in this section. In section 3, we establish the local well-posedness of the phase-field system. Section 4  demonstrates the uniqueness of the inverse problems (IP1) and (IP2), respectively. The paper is concluded in section 5 with a discussion on the future work.

\section{Mathematical formulations and main results}
We focus on the Cahn-Hilliard-Allen-Cahn system for modeling the phase morphology evolution in the Ni-Al binary alloys~\cite{GLF10,K83}. The system is defined in a domain $\Omega=[-1,1]^d$, with  $d=1,2\mbox{ or } 3$, in the form of
\begin{equation}
\label{equ:phase field equation}
\left\{
\begin{aligned}
&\partial_t c(x,t)=-M\alpha \Delta^2c(x,t)+M\Delta \left(\frac{\delta F}{\delta c}\right),\ &&\text{in}\ Q,\\
&\partial_t \eta_i(x,t)=L\beta \Delta \eta_i(x,t)-L\frac{\delta F}{\delta \eta_i(x,t)},\ i=1,2,3, &&\text{in}\ Q,\\
&c(x,0)=\varphi_0(x),&&\text{on}\ \Omega,\\
&\eta_i(x,0) = \varphi_i(x),\ i=1,2,3, &&\text{on}\ \Omega,
\end{aligned}
\right.
\end{equation}
 where $Q = \Omega\times [0,T]$, $\Delta $ is the Laplacian operator with respect to the $x$ variable and $\frac{\delta F}{\delta c}$ means the  variation of $F(c)$ with respect to $c(x,t)$.  In this system, the variable $c$ represents the concentration field, indicating the presence of the so called $\gamma'$ precipitate phase~\cite{ZZ04}. The quantities $\eta_1,\eta_2$, and $\eta_3$, represent three long-range ordered parameters that characterize the ordering of the $\gamma'$ phase. These variables are subject to periodic boundary conditions (PBC) at the boundary $\Sigma = \partial\Omega\times [0,T]$. The parameter $L$ represents the kinetic constant, $M$ represents the diffusion mobility, and $\alpha$ and $\beta$ represent the gradient energy coefficients for $c$ and $\eta_i$, respectively.

The evolution process of alloys is driven by the energy potential functional $F$, which can be decomposed as
\begin{align*}
F(x,t,c,\overrightarrow{\eta})=\int_{\Omega}f_{\text{local}}(c,\overrightarrow{\eta})+f_{\text{elas}}(x,t,c)+f_{\text{pl}}(x,t,c)\mathrm{d}x.
\end{align*}
The local chemical free energy density $f_{\text{local}}$ is typically expressed by Landau polynomial \cite{GLF10}. The elastic strain energy $f_{\text{ela}}$ and plastic strain energy $f_{\text{pl}}$ are derived from the theory of micro-mechanical elasticity and dislocation slip, respectively, and both are dependent on the concentration $c$~\cite{GLF10,KST95}. The initial values $\varphi_0(x)$ and $\varphi_{i}(x)$ are given as small perturbations under periodic conditions. In general, finding an explicit expression for elastic and plastic strain energy is challenging due to the complex interactions  involved in the energy potential function.  Various numerical methods and approximation techniques were proposed to estimate the variations in elastic and plastic strain energy~\cite{GAR20,YWZ17}. However, the development of a comprehensive theory for determining these nonlinear energy functions is still missing. This paper is devoted to partially address this gap by investigating the uniqueness theory.

To facilitate our analysis, let us introduce some periodic functional spaces. For any $1\leq k\in \mathbb{N}$, denote $L^k_p(\Omega)$, $H^k_p(\Omega)$ and $C^k_p(\Omega)$ the functions satisfy PBC on $\Sigma$ associated with the norm
\begin{align*}
	\begin{cases}
\|u\|_{L^k_p(\Omega)}=\|u\|_{L^k(\Omega)}=\left(\int_{\Omega}|u|^k\mathrm{d}x\right)^{1/k},\\ \|u\|_{H^k_p(\Omega)}=\|u\|_{H^k(\Omega)}=\left(\sum\limits_{|\alpha|\leq k}\int_{\Omega} |D^{\alpha} u|^2\mathrm{d}x\right)^{1/2},\\
\|u\|_{C^k_p(\Omega)} = \|u\|_{C^k(\Omega)}=\sum\limits_{|\alpha|\le k}\sup\limits_{x\in\Omega} |D^{\alpha}u|.
\end{cases}
\end{align*}
For periodic vector functions such as the gradient of $u$, we define the norm as \begin{align*}
\|\nabla u\|_{L^k_p(\Omega)}=\|\nabla u\|_{L^k(\Omega)}=\left(\int_{\Omega}|\nabla u\cdot \nabla u|^{k/2}\mathrm{d}x\right)^{1/k}.
\end{align*}
For brevity, we also denote $\mathcal{L}^k = L^2(0,T;H_p^{k}(\Omega))$ with the norm $\|u\|_{\mathcal{L}^k}=\left(\int_0^T \|u\|_{H_p^k(\Omega)}^2 \mathrm{d}t\right)^{1/2}$.

To extend the applicability of our methods to other systems with different forms of chemical energy \cite{BJW01}, we express \eqref{equ:phase field equation} in a more general form
\begin{equation}
\label{equ:nonlinear system}
    \left\{
    \begin{aligned}
        &\partial_t u_0=-c_1\Delta^2u_0+\Delta \left(f_0(x,t,u_0,u_1,u_2,u_3)+g(x,t,u_0)\right),\ &&\text{in}\ Q,\\
        &\partial_t u_i = c_2\Delta u_i+f_i(x,t,u_0,u_1,u_2,u_3),\ &&\text{in}\ Q,\\
        &u_0(x,0) = \varphi_0(x), &&\text{in}\ \Omega,\\
        &u_i(x,0) = \varphi_i(x), &&\text{in}\ \Omega,\\
        &u_0(x,t),u_i(x,t)\ \text{satisfy PBC}, &&\text{on}\ \Sigma,
    \end{aligned}
    \right.
\end{equation}
 where $i=1,2,3$, and $c_1$ and $c_2$ are positive constants. Correspondingly, $u_0$ is the concentration field and $(u_1,u_2,u_3)$ are three long-range ordered parameters. The nonlinear functions $g$ and $f_i$ are variations of the energy operators with respect to $u_0$ and $u_i$.  Let $\overrightarrow{\varphi}=(\varphi_0,\varphi_1,\varphi_2,\varphi_3)$ and $\overrightarrow{u}=(u_0,u_1,u_2,u_3)$ and rewrite $f_i(x,t,u_0,u_1,u_2,u_3)$ as $f_i(x,t,\overrightarrow{u})$ for $i=0,1,2,3$. Besides the periodicity requirement on $\Sigma$ for $g$ and $f_i$, we impose the following admissible conditions on these nonlinear functionals for further analysis.
\begin{definition}[Admissible Conditions]
\label{def:admissible 1}
\
\begin{itemize}
\item[1.] The function $g(x,t,y):Q\times \mathbb{R}\to \mathbb{R}$ is analytic with respect to $(x,t)\in Q$ for all $y\in \mathbb{R}$. Moreover, there exist analytic coefficients $g^{(\ell)}(x,t)$ such that  
\begin{align}
\label{equ:expansion g}
    g(x,t,y) = \sum\limits_{\ell=1}^{\infty}\frac{g^{(\ell)}(x,t)}{\ell!}y^\ell.
\end{align}
\item[2.] The nonlinear function $f_0(x,t,\overrightarrow{z})$ can be written as:
\begin{align}\label{condition2}
f_0(x,t,\overrightarrow{z}) = \sum\limits_{s=1}^{\infty}\frac{1}{s!}\sum\limits_{l_0+l_1+l_2+l_3=s} c_{l_0l_1l_2l_3}(x,t)z_0^{l_0}z_1^{l_1}z_2^{l_2}z_3^{l_3},
\end{align}
where $\overrightarrow{z}=(z_0,z_1,z_2,z_3)$ and the coefficients $c_{l_0l_1l_2l_3}(x,t)$ are analytic in $Q$. For the case $l_0+l_1+l_2+l_3 = 1,$ only $c_{1000}(x,t)\not\equiv 0$.

\item[3.] 
For $i=1,2,3$, the functions $f_i(x,t,\overrightarrow{z}):Q\times \mathbb{R}^4\to \mathbb{R}$ can be expanded as
\begin{align}
    f_i(x,t,\overrightarrow{z}) = b_i(x,t)z_i+\sum\limits_{s=2}^{\infty} \frac{1}{s!}\sum\limits_{l_0+l_1+l_2+l_3=s} b_{i,l_0l_1l_2l_3}(x,t)z_0^{l_0}z_1^{l_1}z_2^{l_2}z_3^{l_3},
\end{align}
where $b_i(x,t)$ and $b_{i,l_0l_1l_2l_3}(x,t)$ are analytic in $Q$ and $b_{i,l_0000}\equiv 0$ for all $l_0\in \mathbb{N}$. 

\item[4.] The functions $g(x,t,y)$ and $f_i(x,t,\overrightarrow{z})$, $i = 0, 1, 2, 3$, are Lipschitz continuous with respect to $y$ and $\overrightarrow{z}$. Namely, there exist constants $L$ and $L_i$, $i=0,1,2,3$, independent of $x$ and $t$ such that:
\begin{align*}
	\begin{cases}
    |g(x,t,y)-g(x,t,0)|\leq L|y|,\\
    |f_i(x,t,\overrightarrow{z})-f_i(x,t,0)|\leq L_i|\overrightarrow{z}|,\ i=0,1,2,3.
    \end{cases}
\end{align*}


\end{itemize}
\end{definition}
\begin{remark}
\label{rmk:admissible condition}
The admissible conditions 1-4 are consistent with many existing phase-field models. For example, the commonly used double well potential satisfies these requirements \cite{GLF10}. These conditions  also ensure that under the zero initial condition, equation \eqref{equ:nonlinear system} admits a unique zero solution. Proof is provided in the appendix. 
\end{remark}
Given all these preparations, the forward problem  can be described as follows:
\begin{itemize}
\item[(P1)] Assuming the nonlinear functions $f_i$, $i=0,1,2,3$, and $g$ satisfy the admissible conditions \ref{def:admissible 1},  find $\overrightarrow{u}$ that solves system (\ref{equ:nonlinear system}) with a given initial condition $\overrightarrow{\varphi}$.
\end{itemize}
 We establish the following local well-posedness result for the forward problem (P1).
\begin{theorem}[Local well-posedness of (P1)] 
	\label{theo:LWP}
	Under the admissible conditions \ref{def:admissible 1},  for an even integer $q > d/2$,  assume the initial condition $\overrightarrow{\varphi}$  belongs to
	\begin{align*}
		\mathcal{N}_{\delta}=\left\{\overrightarrow{\varphi}\in H_p^{q+2}(\Omega)\times \left( H_p^{q+1}(\Omega)\right)^3:\|\varphi_0\|_{H_p^{q+2}(\Omega)}+\sum\limits_{i=1}^3\|\varphi_i\|_{H_p^{q+1}(\Omega)}\leq \delta \right\},
	\end{align*}
	with $\delta>0$ sufficiently small.
	Then the solution to (P1) is given by $\frac{d^ku_0}{dt^k} \in \mathcal{L}^{q+4-2k}$ and $\frac{d^ku_i}{dt^k} \in \mathcal{L}^{q+2-2k}$ for $i=1,2,3$ and $k=0,\cdots,q/2$. Moreover, the solution satisfies the energy estimate

\begin{align*}
\sum\limits_{k=0}^{q/2}\left(\left\|\frac{d^ku_0}{dt^k}\right\|_{\mathcal{L}^{q+4-2k}}+\sum\limits_{i=1}^3\left\|\frac{d^ku_i}{dt^k}\right\|_{\mathcal{L}^{q+2-2k}}\right)\leq C\delta,
\end{align*}
	where $C$ is a positive constant that depends only on $\Omega$ and $T$. 
\end{theorem}
Note that here `local' means the solution to (P1) is only guaranteed when the initial conditions are restricted to $\mathcal{N}_{\delta}$.

 For inverse problems, it is generally difficult to determine \( g \) and \( f_i \), \(i=0,1,2,3\), simultaneously. In this paper, assuming \( f_i \), \(i=0,1,2,3\), are known, we are interested in recovering the nonlinear term \( g(x,t,u_0) \) from measurements of the concentration field \( u_0 \) at single or multiple time points. We define the single-shot measurement \( \overrightarrow{M}_s(g,\overrightarrow{\varphi}) \) and the multi-shot measurement \( \overrightarrow{M}_m(g,\overrightarrow{\varphi}) \) as
\begin{equation}
\begin{aligned}
	\begin{cases}
\overrightarrow{M}_s(g,\overrightarrow{\varphi}) = [u_0(\cdot,t_1),\partial_t u_0(\cdot,t_1)], \\
\overrightarrow{M}_m(g,\overrightarrow{\varphi}) = [u_0(\cdot,t_1),\partial_t u_0(\cdot,t_1),\cdots,u_0(\cdot,t_N),\partial_t u_0(\cdot,t_N)],
\end{cases}
\end{aligned}
\end{equation}
where $u_0$ is the solution to equation (\ref{equ:nonlinear system}) under energy function $g$ and the initial conditions $\overrightarrow{\varphi}$, and $t_1,\cdots,t_N$ are distinct time points within the interval $(0,T]$. We reformulate the two inverse problems mentioned in the introduction as:
\begin{itemize}
\item (IP1): Assume $g_1$ and $g_2$ are independent of time, namely, $g_j=g_j(x,y),$ $j=1,2$. For any initial value $\overrightarrow{\varphi}$, if there holds
\begin{align}
\overrightarrow{M}_s(g_1,\overrightarrow{\varphi}) = \overrightarrow{M}_s(g_2,\overrightarrow{\varphi}),
\end{align}
does it imply $g_1(x,y) = g_2(x,y)$?
\item (IP2):  Assume $g_1$ and $g_2$ are dependent of time, namely, $g_j=g_j(x,t,y),$ $j=1,2$.  For any initial value $\overrightarrow{\varphi}$, if there holds
\begin{align}
\overrightarrow{M}_m(g_1,\overrightarrow{\varphi}) = \overrightarrow{M}_m(g_2,\overrightarrow{\varphi}),
\end{align}
does it imply  $g_1(x,t,y) = g_2(x,t,y)$?
\end{itemize}

\begin{remark}
Since $\partial_t^k u_0(x,t)\in \mathcal{L}^{q}$ for $k=0,1,2$ in Theorem \ref{theo:LWP},
then by the Sobolev embedding theorem \cite{A03}, when $q>d/2$, we have $u_0\in C^0(0, T; C^0_p(\Omega)),\partial_t u_0\in C^0(0, T; C^0_p(\Omega))$, which ensures the validity of the subsequent inverse problem measurements.
\end{remark}

We give a positive answer to (IP1), which is stated as follows:
\begin{theorem}[Uniqueness of (IP1)]
\label{theo:IP1}
Assume that $g_j(x,y)$, $j=1,2$, are admissible functions  with the form
\begin{align}\label{g_expan}
g_j(x,y) = \sum\limits_{\ell=1}^{\infty} \frac{g_j^{(\ell)}(x)}{\ell!}y^\ell,
\end{align}
and the initial value $\overrightarrow{\varphi}$ satisfies the condition in Theorem \ref{theo:LWP}. Let $\overrightarrow{M}_s(g_j,\overrightarrow{\varphi})$ be the measurement of the semilinear system
\begin{equation}
	\label{semilinear_sys}
    \left\{
    \begin{aligned}
        &\partial_t u_{j,0}=-c_1\Delta^2u_{j,0}+\Delta \left(f_0(x,t,\overrightarrow{u_{j}})+g_j(x,u_{j,0})\right),\ &&\text{in}\ Q,\\
        &\partial_t u_{j,i} = c_2\Delta u_{j,i}+f_i(x,t,\overrightarrow{u_{j}}),\ i=1,2,3,&&\text{in}\ Q,\\
        &u_{j,0}(x,0) = \varphi_0(x),&&\text{in}\ \Omega,\\
        &u_{j,i}(x,0) = \varphi_i(x),\ i=1,2,3,&&\text{in}\ \Omega,\\
        &u_{j,0}(x,t),u_{j,i}(x,t)\ \text{satisfy PBC}, &&\text{on}\ \Sigma,
    \end{aligned}
    \right.
\end{equation}
with $j=1,2$. If
\begin{align*}
\overrightarrow{M}_s(g_1,\overrightarrow{\varphi}) = \overrightarrow{M}_s(g_2,\overrightarrow{\varphi})
\end{align*}
for any $\overrightarrow{\varphi}$, then $g_1^{(\ell)}(x) = g_2^{(\ell)}(x)$, $\ell\ge 1$, almost everywhere in $\Omega$. 
\end{theorem}
Regarding the inverse problem (IP2), due to the limited measurements available at finitely many time points, we can only infer the local properties of the function $g(x,t,y)$ during the inversion process. Specifically, the result for (IP2) is stated as follows:

\begin{theorem}[Uniqueness of (IP2)]
\label{theo:IP2}
Assume $g_j(x,t,y)$, $j=1,2$, are admissible functions given in the form of equation \eqref{equ:expansion g}
and the initial value $\overrightarrow{\varphi}$ satisfy the condition in Theorem \ref{theo:LWP}. Let $\overrightarrow{M}_m(g_j,\overrightarrow{\varphi})$  be the measurement of the semilinear system \eqref{semilinear_sys}. If
\begin{align*}
\overrightarrow{M}_m(g_1,\overrightarrow{\varphi}) = \overrightarrow{M}_m(g_2,\overrightarrow{\varphi})
\end{align*}
for any $\overrightarrow{\varphi}$, then $g_1^{(\ell)}(x,t_j)=g_2^{(\ell)}(x,t_j)$, $ \ell \ge 1$, almost everywhere in $\Omega$ for $t_1,\cdots,t_N\in (0,T]$, and
\begin{align*}
    \sup\limits_{t\in [0,T]}\left|g_1^{(\ell)}(x,t)-g_2^{(\ell)}(x,t)\right|\leq \frac{C_{\ell}}{(N+1)!},\ \forall (x,t)\in Q,
\end{align*}
with $C_\ell>0$ depending on $T$ and $\|\partial^{N+1}_t g^{(\ell)}_j\|_{L^{\infty}(Q)}$ for $j=1,2$.
\end{theorem}
\begin{remark}
Our result can be extended to the case when $c_1$ and $c_2$ are positive functions on $Q$. We omit the details for conciseness.
\end{remark}

\section{Proof of local well-posedness on (P1)}
In this section, we establish the local well-posedness of  system (\ref{equ:nonlinear system}) by utilizing the implicit function theorem in Banach spaces \cite{E12}. 

\begin{lemma}[Implicit Function Theorem for Banach Spaces]
\label{lemma:implicit function}
Let $X$, $Y$, and $Z$ be Banach spaces and suppose the mapping $f:X\times Y\to Z$ is continuously Fréchet differentiable. If $(x_0,y_0)\in X\times Y$ satisfies $f(x_0,y_0)=0$ and the map $y\mapsto D_yf(x_0,y_0)(0,y)$ is a Banach space isomorphism from $Y$ to $Z$, then there exist neighborhoods $U$ of $x_0$ and $V$ of $y_0$, and a Fréchet differentiable function $g:U\to V$ such that $f(x,g(x))=0$ if and only if $y=g(x)$, for all $(x,y)\in U\times V$.
\end{lemma}

We also require the following lemma to prove the analyticity of a map, as described in \cite{P87}.

\begin{lemma}
\label{lemma:analytic}
Let $f: U \to F$ be a mapping, $U$ be a subset of a complex Banach space $E$, and $F$ be a complex Banach space. The following two conditions are equivalent:
\begin{enumerate}
\item $f$ is analytic on $U$.
\item $f$ is locally bounded and weakly analytic on $U$, where weakly analytic means that for any $x_0, x \in U$, the mapping
\[
\lambda \mapsto f(x_0 + \lambda x)
\]
is analytic in a neighborhood of the origin in the complex plane.
\end{enumerate}
\end{lemma}
We first recall the global well-posedness of linear parabolic equations in \cite{L10} and then prove the global well-posedness of the linear Cahn-Hillard equation under the PBC.

\begin{lemma}[Global Well-posedness of Linear Parabolic Equations]
\label{lemma:WP of Linear parabolic equation}
Consider the initial-boundary value problem for a parabolic equation given by
\begin{equation*}
\left\{
\begin{aligned}
&u_{t}+\mathcal{L}u=p(x,t), &&\text{in}\ Q,\\
&u(x,0)=\varphi(x), &&\text{on}\ \Omega,\\
&u\ \text{satisfies PBC}, &&\text{on}\ \Sigma,
\end{aligned}
\right.
\end{equation*}
where $\mathcal{L}$ denotes a second order partial differential operator defined as
\begin{align*}
\mathcal{L}u := -\sum\limits_{i,j=1}^na_{ij}(x)u_{x_ix_j}+\sum\limits_{i=1}^nb_i(x,t)u_{x_i}+c(x,t)u.
\end{align*}
The coefficient $a_{ij}(x)$ is uniformly elliptic and symmetric, and $a_{ij}$, $b_i$, $c \in C_p^{\infty}(Q)$. The source term and initial value satisfy $\varphi \in H_p^{q+1}(\Omega),\frac{\mathrm{d}^kp}{\mathrm{d} t^k}\in \mathcal{L}^{q-2k}:=L^2(0,T;H_p^{q-2k}(\Omega))$ for $k=0,\dots,q/2,$ with an even integer $q>0$, and the compatibility conditions hold:
 \begin{equation*}
 \varphi_0:=\varphi \in H_p^1(\Omega), \varphi_1:=p(0)-\mathcal{L}\varphi_0\in H_p^1(\Omega),\cdots,\varphi_q:=\frac{d^{q-1}p}{dt^{q-1}}(0)-\mathcal{L}\varphi_{q-1}\in H_p^1(\Omega).
 \end{equation*}
Then there exists a unique solution $u$ satisfying $\frac{d^k u}{dt^k}\in \mathcal{L}^{q+2-2k}$ for $k=0,\cdots,q/2$.
\end{lemma}

To prove the well-posedness of the linear Cahn-Hilliard equation, we equivalently reformulate it as a coupling of two second order equations, and show the following theorem.

\begin{lemma}[Global Well-posedness of the Linear Cahn-Hilliard Equation]
		\label{lemma: WP of Linear CH equation}
		Consider the initial-boundary value problem for the reformulated linear Cahn-Hilliard equation given by
		\begin{equation}
			\label{equ:linear CH}
			\left\{
			\begin{aligned}
				&u_{t}+a\Delta v-\Delta (b(x,t)u)=p(x,t), &&\text{in}\ Q,\\
				&v = \Delta u,&&\text{in}\ Q,\\
				&u(x,0)=\psi(x),\ v(x,0) = \Delta \psi(x), &&\text{on}\ \Omega,\\
				&u,\ v\ \text{satisfy the PBC}, &&\text{on}\ \Sigma.
			\end{aligned}
			\right.
		\end{equation}
		with coefficients satisfying $a>0$, $b\in L^{\infty}(0,T;H_p^1(\Omega))$. The source term $p$ and the initial value $\psi$ are assumed to be $\psi\in H_p^2(\Omega),p\in L^{2}(0,T;L_p^{\infty}(\Omega))$. Then there exists a unique solution $\begin{pmatrix}u\\v\end{pmatrix}$ to \eqref{equ:linear CH} satisfying		
		\begin{align*}
			\partial_t u\in L^2(0,T;H_p^{-1}(\Omega)),\ u\in L^{\infty}(0,T;H_p^1(\Omega)),\ 
			v\in L^2(0,T;H_p^1(\Omega)),
		\end{align*}
where $H_p^{-1}(\Omega)$ is the dual space of $H_p^1(\Omega)$, and the energy estimate holds
		\begin{align*}
			\|\partial_t u\|_{L^2(0,T;H_p^{-1}(\Omega))} +\|u\|_{L^{\infty}(0,T;H_p^1(\Omega))}+\|v\|_{L^2(0,T;H_p^1(\Omega))}\leq C,
		\end{align*}
		where $C$ is determined by $\Omega,T,a,\|b\|_{L^{\infty}(0,T;H_p^1(\Omega))},\|\psi\|_{H_p^2(\Omega)}$ and $\|p\|_{L^{2}(0,T;L_p^{\infty}(\Omega))}$.
\end{lemma}

\begin{remark}
For convenience in subsequent discussions, we  assume the constant $a = 1$, otherwise we can replace the variable $x$ by $x^* = x/\sqrt[4]
{a}$. 
\end{remark}

\begin{remark}
	\label{def:weak solution}
 The weak solution of equation (\ref{equ:linear CH}) is given as follows
	\begin{align}
		\label{equ:weak equation}
		\begin{cases}
			\langle \partial_tu,w\rangle=(\nabla v,\nabla w)-(\nabla (b(x,t)u),\nabla w)+(p,w),\\
			(v,w)+(\nabla u,\nabla w)=0,
		\end{cases}
	\end{align}
	for each $w\in H_p^1(\Omega)$ and almost every time $t\in [0,T]$, where $(\cdot,\cdot)$ denotes the inner product in $L_p^2(\Omega)$ and $\langle\cdot,\cdot\rangle$ denotes the inner product in the dual space  of $H_p^1(\Omega)$ and $H_p^{-1}(\Omega)$.
\end{remark}

To prove Lemma \ref{lemma: WP of Linear CH equation}, we utilize the Galerkin method \cite{L10} to construct the weak solution of equation \eqref{equ:linear CH} and establish its existence, uniqueness, and stability through the energy estimate.

\begin{proof}[Proof of Lemma \ref{lemma: WP of Linear CH equation}:] The proof is divided into three steps.

\textbf{Step 1:} 
	Let $\{\omega_k\}_{k=1}^{\infty}$ be a set of smooth functions that form an orthogonal basis in $H_p^1(\Omega)$  with $\omega_1=1$.  For instance, in case when $\Omega=[-1,1]$, we can choose the Fourier basis $\{\omega_k\}_{k=1}^{\infty} =\{\{\cos(2k\pi x)\}_{k=0}^{\infty},\{\sin(2k\pi x)\}_{k=1}^{\infty}\}$.

Consider functions $u_m,v_m:[0,T]\to H_p^1(\Omega)$ defined as follows
	\begin{equation}\label{uvexpan}
		u_m(t)=\sum_{k=1}^m \alpha_{m,k}(t)\omega_k,\quad v_m(t)=\sum_{k=1}^m \beta_{m,k}(t)\omega_k.
	\end{equation}
	These functions satisfy the equations

	\begin{align}
		\label{equ:weak form of um}
		\begin{cases}
		(\partial_tu_m,\omega_k)=(\nabla v_m,\nabla \omega_k)-(\nabla \left(b(x,t)u_m\right),\nabla \omega_k)+(p,\omega_k),\\
		(v_m,\omega_k)+(\nabla u_m,\nabla \omega_k)=0,
		\end{cases}
	\end{align}
for all $\omega_k$. Here, the coefficients $\alpha_{m,k}(t)$ and $\beta_{m,k}(t)$ are determined by the ordinary differential equations and initial values with
 $u_m(0) = \mathcal{P}_m(\psi(x))$ and $v_m(0) = \mathcal{P}_m(\Delta \psi(x))$, where $\mathcal{P}_m$ is the projection operator onto the space $W_m = \mbox{span}(\omega_1,\cdots,\omega_m)$. The Cauchy–Peano theorem \cite{PGG90} and the smoothness of $b(x,t)$ allow us to deduce the existence of $\alpha_{m,k}(t)$ and $\beta_{m,k}(t)$.

\textbf{Step 2:} In (\ref{equ:weak form of um}), we take  $v_m$ and $\partial_t u_m$ as test functions, which gives
	\begin{align*}
		\begin{cases}
	(\partial_tu_m,v_m)=\|\nabla v_m\|^2_{L_p^2(\Omega)}-(\nabla \left(b(x,t)u_m\right),\nabla v_m)+(p,v_m),\\
		(v_m,\partial_t u_m)+(\nabla u_m,\partial_t\nabla  u_m)=0.
		\end{cases}
	\end{align*}

Similarly, using $u_m$ and $v_m$ as test functions, we have
\begin{align*}
\begin{cases}
       \frac{1}{2}\frac{d}{dt}\|u_m\|^2_{L_p^2(\Omega)}=(\nabla v_m,\nabla u_m)-(\nabla \left(b(x,t)u_m\right),\nabla u_m)+(p,u_m),\\
		\|v_m\|^2_{L_p^2(\Omega)}+(\nabla u_m,\nabla v_m)=0. 
\end{cases}
\end{align*} 
Combining them yields
	\begin{align*}
		&\frac{1}{2}\frac{d}{dt}\left(\|u_m\|^2_{L_p^2(\Omega)}+\|\nabla u_m\|^2_{L^2_p(\Omega)}\right)+\|v_m\|^2_{L_p^2(\Omega)}+\|\nabla v_m\|^2_{L_p^2(\Omega)}\\
  =&\int_{\Omega}\nabla \left(b(x,t)u_m\right)\cdot \nabla (v_m-u_m)+p(u_m-v_m)\mathrm{d}x.
	\end{align*}
By applying the Cauchy inequality, it holds
\begin{equation}
\label{equ: estimate RHS1}
     \begin{aligned}
     &\int_{\Omega}\nabla \left(b(x,t)u_m\right)\cdot \nabla (v_m-u_m)\mathrm{d}x\\
     \leq & \varepsilon_1\|\nabla \left(b(x,t)u_m\right)\|_{L_p^2(\Omega)}^2+\frac{1}{4\varepsilon_1} \|\nabla (v_m-u_m)\|_{L_p^2(\Omega)}^2\\
     \leq &2\varepsilon_1\left(\|b(x,t)\nabla u_m\|^2_{L_p^2(\Omega)}+\|u_m\nabla \left( b(x,t)\right) \|_{L_p^2(\Omega)}^2\right)+\frac{1}{2\varepsilon_1} \left(\|\nabla v_m\|_{L_p^2(\Omega)}^2+\|\nabla u_m\|_{L_p^2(\Omega)}^2\right)\\
     \leq & 2C_1\varepsilon_1\left(\|\nabla u_m\|_{L_p^2(\Omega)}^2+\| u_m\|_{L_p^2(\Omega)}^2\right)+\frac{1}{2\varepsilon_1} \left(\|\nabla v_m\|_{L_p^2(\Omega)}^2+\|\nabla u_m\|_{L_p^2(\Omega)}^2\right),
 \end{aligned}
\end{equation}
 where $C_1$ is determined by $\|b(x,t)\|_{L^{\infty}(0,T;H_p^1(\Omega))}$, and 
 \begin{equation}
 \label{equ: estimate RHS2}
      \begin{aligned}
     &\int_{\Omega} p(u_m-v_m) \mathrm{d}x\\
     \leq & \varepsilon_2 \| p\|_{L_p^2(\Omega)}^2+\frac{1}{4\varepsilon_2}\|u_m-v_m\|_{L_p^2(\Omega)}^2\\
     \leq &\varepsilon_2 \| p\|_{L_p^2(\Omega)}^2+\frac{1}{2\varepsilon_2}\left(\|u_m\|_{L_p^2(\Omega)}^2+\|v_m\|_{L_p^2(\Omega)}^2\right).
 \end{aligned}
 \end{equation}
 By selecting $\varepsilon_1 =\varepsilon_2 = 2$ in \eqref{equ: estimate RHS1} and \eqref{equ: estimate RHS2}, it yields 
 \begin{equation}
 \label{equ: estimate RHS3}
      \begin{aligned}
     &\frac{d}{dt}\left(\|u_m\|^2_{L_p^2(\Omega)}+\|\nabla u_m\|^2_{L^2_p(\Omega)}\right)+\|v_m\|^2_{L_p^2(\Omega)}+\|\nabla v_m\|^2_{L_p^2(\Omega)}\\\leq & C_2\left(\|u_m\|^2_{L_p^2(\Omega)}+\|\nabla u_m\|^2_{L^2_p(\Omega)} \right)+C_3,
 \end{aligned}
 \end{equation}
where $C_2$ and $C_3$ depend on $\|b(x,t)\|_{L^{\infty}(0,T;H_p^1(\Omega))}$ and $\|p(x,t)\|_{L^{\infty}(0,T;L_p^2(\Omega))}$.  Integrating both sides of \eqref{equ: estimate RHS3}  over $[0,s]$, we get
	\begin{align*}
		&\|\nabla u_m(s)\|^2_{L^2_p(\Omega)}+\| u_m(s)\|^2_{L^2_p(\Omega)}+\|v_m\|^2_{L^2(0,s;L_p^2(\Omega))}+\|\nabla v_m\|^2_{L^2(0,s;L_p^2(\Omega))}\\\
  \leq &C_2\int_0^s \left(\|u_m\|^2_{L_p^2(\Omega)}+\|\nabla u_m\|^2_{L^2_p(\Omega)}\right)\mathrm{d}t+C_3s.
	\end{align*}
	By the Gronwall inequality, there exists $C_4$ such that:
	\begin{align}
 \label{equ:u_m_bounded}
		\|u_m\|_{L^{\infty}((0,T),H_p^1(\Omega))}+\|v_m\|_{L^2(0,T;H_p^1(\Omega))}\leq C_4.
	\end{align}

Based on the estimate above, by fixing any $w\in H^1_p(\Omega)$ with $\|w\|_{H_p^1(\Omega)}\leq 1$, we can deduce that 
 \begin{align*}
     \langle \partial_t u_m,w \rangle &= (\partial_t u_m,w)=(\partial_t u_m,\mathcal{P}_m(w))\\
 &=(\nabla v_m,\nabla\mathcal{P}_m( w))-(\nabla (\mathcal{P}_m\left(b(x,t)\right)u_m),\nabla \mathcal{P}_m(w))+(\mathcal{P}_m(p),\mathcal{P}_m(w)).
 \end{align*}
 Consequently
 \begin{align*}
     |\langle \partial_t u_m,w \rangle|\leq C_5\left(\|p\|_{L_p^2(\Omega)}+\|u_m\|_{H_p^1(\Omega)}+\|v_m\|_{H_p^1(\Omega)}\right),
 \end{align*}
 since $\|\mathcal{P}_m(w)\|_{H^1(\Omega)}\leq \|w\|_{H^1(\Omega)}\leq 1$. Thus
 \begin{align*}
     \|\partial_t u_m\|_{H_p^{-1}(\Omega)}\leq C_5\left(\|p\|_{L_p^2(\Omega)}+\|u_m\|_{H_p^1(\Omega)}+\|v_m\|_{H_p^1(\Omega)}\right),
 \end{align*}
 and therefore 
\begin{align*}
    \|\partial_t u_m\|_{L^2(0,T;H_p^{-1}(\Omega))}\leq C_6\left(\|p\|_{L^2(0,T;L_p^2(\Omega))}+\|u_m\|_{L^2(0,T;H_p^1(\Omega))}+\|v_m\|_{L^2(0,T;H_p^1(\Omega))}\right).
\end{align*}

\textbf{Step 3:} According to the energy estimates for $u_m$, $v_m$ and $\partial_t u_m$, we have the existence of a subsequence satisfying:
	\begin{align*}
		\begin{cases}
		&u_m\rightharpoonup u\ \text{in}\ L^{\infty}(0,T;H_p^1(\Omega)),\\ &\partial_t u_m\rightharpoonup \partial_t u\ \text{in}\ L^2(0,T;H^{-1}_p(\Omega)),\\
		&v_m\rightharpoonup v\ \text{in}\ L^2(0,T;H_p^1(\Omega)),
		\end{cases}
	\end{align*}
where $u,\partial_t u,v$ satisfy the weak form \eqref{equ:weak equation} and initial value $u(x,0) = \psi(x),v(x,0) = \Delta \psi(x)$. Since the subsequent proof follows the standard technique (e.g., as discussed in \cite{L10} for parabolic equations), the proof is completed.
\end{proof}
When the coefficient $b$ has a higher order regularity, we can establish the existence and uniqueness of the global strong solution to equation (\ref{equ:linear CH}). 
	\begin{lemma}
		\label{lemma:linear CH inprove regularity}
		Under the assumptions of Lemma \ref{lemma: WP of Linear CH equation}, if we further assume \( b \in L^{\infty}(0,T;H^2_p(\Omega)),\) then equation (\ref{equ:linear CH}) has a unique global strong solution \( \begin{pmatrix} u \\ v \end{pmatrix} \) that satisfies the following conditions:
\begin{align*}
  &u \in L^{\infty}(0,T;H_p^2(\Omega)) \cap L^2(0,T;H_p^4(\Omega)),\,\,\, \partial_t u \in L^{\infty}(0,T;L_p^2(\Omega)) \cap L^2(0,T;H_p^2(\Omega)), \\
  &\partial_t^2 u\in L_p^2(0,T;H^{-1}(\Omega)),\,\,\, v \in L^2(0,T;H_p^2(\Omega)).  
\end{align*}

\end{lemma}
\begin{proof}
		First, we take $\{\omega_k\}_{k=1}^{\infty}$ as an orthogonal basis in $H_p^4(\Omega)$ for the expansions $u_m$ and $v_m$ in equation \eqref{uvexpan}. Then consider the elliptic equation under periodic boundary conditions
		\begin{align}
			\label{equ:ellipse equation}
		 \Delta u_m(\cdot,t)=	v_m(\cdot,t)\ \text{in}\ \Omega,\ \forall t\in [0,T].
		\end{align}
		According to the regularity theory for elliptic equations, it holds $u_m\in L^{2}(0,T;H_p^3(\Omega))$  since $v_m\in L^2(0,T;H_p^1(\Omega))$ by the proof of Lemma \ref{lemma: WP of Linear CH equation}. Further, considering $\Delta^2 u_m$ as a test function in (\ref{equ:weak form of um}), we have
\begin{align}
\label{equ:High regular weak form}
    (\partial_t u_m,\Delta^2 u_m) = -(\Delta v_m,\Delta^2 u_m)+(\Delta (b(x,t)u_m), \Delta^2 u_m)+(p,\Delta^2 u_m).
		\end{align}
		Substituting (\ref{equ:ellipse equation}) into \eqref{equ:High regular weak form} and applying the integration by parts, we get
		\begin{align*}
			\frac{1}{2}\frac{d}{dt}\int_{\Omega} (\Delta u_m)^2\mathrm{d}x+\|\Delta^2 u_m\|^2_{L_p^2(\Omega)}=R_1+R_2,
		\end{align*}
		where $R_1 = (\Delta (b(x,t)u_m), \Delta^2 u_m),R_2 = (p,\Delta^2 u_m)$. By the Cauchy inequality, it yields
		\begin{align*}
			R_1 &\leq \frac{1}{4}\|\Delta^2 u_m\|_{L_p^2(\Omega)}^2+C_7\|\Delta u_m\|_{L_p^2(\Omega)}^2+C_7,\\
			R_2 &\leq \frac{1}{4} \|\Delta^2 u_m\|_{L_p^2(\Omega)}^2+C_7,
		\end{align*}
		where $C_7$ depends on $\|b\|_{L^{\infty}(0,T;H_p^2(\Omega))}$, $\|u_m\|_{L^{\infty}(0,T;H_p^1(\Omega))}$ and $\|p\|_{L^2(0,T;L_p^2(\Omega))}$. Since $\|u_m\|_{L^{\infty}(0,T;H_p^1(\Omega))}$ is bounded by the inequality \eqref{equ:u_m_bounded}, we have
		\begin{align}
			\label{equ:weak form}
			\frac{d}{dt}\int_{\Omega} (\Delta u_m)^2\mathrm{d}x+\|\Delta^2 u_m\|^2_{L_p^2(\Omega)}\leq 2C_8\|\Delta u_m\|_{L_p^2(\Omega)}^2+2C_8.
		\end{align}
		Integrating (\ref{equ:weak form}) over $[0,s]$ yields
		\begin{align*}
			\|\Delta u_m(s)\|_{L_p^2(\Omega)}^2+\|\Delta^2 u_m\|^2_{L^2(0,s;L_p^2(\Omega))}&\leq 2C_8\int_0^s\|\Delta u_m\|_{L_p^2(\Omega)}^2\mathrm{d}t+2C_8s+\|\mathcal{P}_m(\Delta \psi(x))\|_{L_p^2(\Omega)}^2.
		\end{align*} 
  Using the integral form of Gronwall inequality and considering the supremum over $[0,T]$, there exists $C_9>0$, such that
		\begin{align*}
			\|\Delta u_m\|_{L^{\infty}(0,T;L_p^2(\Omega))}+\|\Delta^2 u_m\|_{L^2(0,T;L_p^2(\Omega))}\leq C_9.
		\end{align*}
		By combining the estimate above, there exists a constant $C_{10}$ such that
		\begin{align*}
			\|u_m\|_{L^{\infty}(0,T;H_p^2(\Omega))\cap L^2(0,T;H_p^4(\Omega))}\leq C_{10}.
		\end{align*}
Since \(\Delta v_m = u_m\), it follows that \(v_m \in L_2(0, T; H_p^2(\Omega))\). We now proceed to estimate \(\partial_t u_m\) and \(\partial_t^2 u_m\). By using \(\partial_t u_m\) as a test function in the weak formulation of \(\eqref{equ:weak form of um}\), we obtain
\begin{align*}
    (\partial_t u_m, \partial_t u_m) + (\Delta u_m, \Delta (\partial_t u_m)) = (p, \partial_t u_m) + (\Delta (b(x, t) u_m), \partial_t u_m).
\end{align*}
This leads to the inequality
\begin{align*}
    \|\partial_t u_m\|^2_{L_p^2(\Omega)} + \frac{1}{2} \frac{d}{dt} \|\Delta u_m\|^2_{L_p^2(\Omega)} \leq \|p\|^2_{L_p^2(\Omega)} + C_{11}\|u_m\|^2_{H_p^2(\Omega)} + \frac{1}{2}\|\partial_t u_m\|^2_{L_p^2(\Omega)},
\end{align*}
where the constant \(C_{11}\) depends on \(\|b\|_{L^{\infty}(0,T;H_p^2(\Omega))}\). As a result, we conclude that \(\partial_t u_m \in L^2(0, T; L_p^2(\Omega))\).

Next, by differentiating \(\eqref{equ:weak form of um}\) with respect to time \(t\) and using \(\partial_t u_m\) as the test function, we derive that
\begin{align*}
    \left(\partial_t^2 u_m, \partial_t u_m\right) + (\Delta \partial_t u_m, \Delta \partial_t u_m)  = (\partial_t p, \partial_t u_m)+ \left(u_m\partial_t b + b\partial_t u_m,  \Delta  \partial_t u_m\right).
\end{align*}
It holds the following inequality
\begin{align*}
    &\frac{1}{2} \frac{d}{dt} \|\partial_t u_m\|^2_{L_p^2(\Omega)} + \|\Delta \partial_t u_m\|^2_{L_p^2(\Omega)} \\
    \leq &\frac{1}{2} \left(\|\partial_t p\|_{L_p^2(\Omega)}^2 + \|\partial_t u_m\|_{L_p^2(\Omega)}^2\right) + \frac{1}{2}\|\Delta \partial_t u_m\|_{L_p^2(\Omega)}^2\\&+\frac{1}{2}\left(\|u_m\|_{L_p^2(\Omega)}^4+ \|\partial_t u_m\|_{L_p^2(\Omega)}^4+\|b\|_{L_p^2(\Omega)}^4+\|\partial_t b\|_{L_p^2(\Omega)}^4\right).
\end{align*}
Therefore, we have
\begin{align*}
    &\frac{d}{dt} \|\partial_t u_m\|^2_{L_p^2(\Omega)} +  \|\Delta \partial_t u_m\|^2_{L_p^2(\Omega)} \\
    \leq & C_{12} \left(1 + \|\partial_t u_m\|^2_{L_p^2(\Omega)}+ \|\partial_t u_m\|^4_{L_p^2(\Omega)} \right),
\end{align*}
where \(C_{12}\) depends on \(\|u_m\|_{L_p^2(\Omega)}, \|\partial_t p\|_{L_p^2(\Omega)}, \|b\|_{L_p^2(\Omega)}, \|\partial_t b\|_{H_p^1(\Omega)}\). Integrating this inequality over $[0,s]$, we obtain
\begin{align*}
\|\partial_t u_m(s)\|^2_{L_p^2(\Omega)}+\|\Delta \partial_t u_m\|^2_{L^2(0,T;L_p^2(\Omega))}\leq C_{12}s+C_{12}\int_0^s \|\partial_t u_m\|^2_{L_p^2(\Omega)}+ \|\partial_t u_m\|^4_{L_p^2(\Omega)} \mathrm{d}t
\end{align*}
According to the extended Gronwall inequality \cite{J18} and taking the supremum over $s$ on $[0,T]$, we deduce that
\begin{align*}
    \|\partial_t u_m\|_{L^{\infty}(0,T;L_p^2(\Omega))}+ \|\Delta \partial_t u_m\|_{L^{2}(0,T;L_p^2(\Omega))}\leq C_{13}.
\end{align*}
Since $\|\partial_t u_m\|_{L^2(0,T;H^2_p(\Omega))}$ is equivalent to $\|\Delta \partial_t u_m\|_{L^2(0,T;L^2_p(\Omega))}$ as in \cite{A03}, we have $\partial_t u_m\in L^2(0,T;H^2_p(\Omega))$.

As in the proof of Theorem \(\ref{lemma: WP of Linear CH equation}\), we consider \(w \in H_p^2(\Omega)\) with \(\|w\|_{H_p^2(\Omega)} \leq 1\). Then, for each \(t \in [0, T]\)
\begin{align*}
    \langle \partial_t^2 u_m, w \rangle = (\partial_t^2 u_m, w) = (\partial_t^2 u_m, \mathcal{P}_m (w)).
\end{align*}
Differentiating \(\eqref{equ:weak form of um}\) with respect to \(t\), we get
\begin{align*}
    \langle \partial_t^2 u_m, w \rangle = (\partial_t^2 u_m, \mathcal{P}_m (w)) = R_3,
\end{align*}
where 
\begin{align*}
    R_3 := (\partial_t p, \mathcal{P}_m (w)) - (\Delta \partial_t u_m, \Delta \mathcal{P}_m (w)) - \left(\nabla \left(u_m \partial_t b + b \partial_t u_m\right), \mathcal{P}_m (\nabla w)\right).
\end{align*}
Thus, we have the following estimate
\begin{align*}
    |\langle \partial_t^2 u_m, w \rangle| \leq |R_3| \leq C_{14} \left(\|\partial_t p\|^2_{L_p^2(\Omega)} + \|\partial_t u_m\|^2_{H_p^2(\Omega)} + \|u_m\|^2_{H_p^2(\Omega)} + \|\partial_t b\|^2_{H_p^1(\Omega)} + \|b\|^2_{H_p^1(\Omega)}\right).
\end{align*}
Based on this estimate, we conclude that
\begin{align*}
    \|\partial_t^2 u_m\|_{L^2(0, T; H_p^{-1}(\Omega))} \leq C_{15}.
\end{align*} 
		
  Letting $m\to \infty$, we deduce that there exists a unique solution \( \begin{pmatrix} u \\ v \end{pmatrix} \) satisfying equation (\ref{equ:linear CH}).
	\end{proof}
	
\begin{remark}
		\label{remark:higher regularity}
		When the source term and initial conditions have a higher order regularity, we can infer that $u$ also possesses a better regularity by induction. In particular, if
$\psi\in H_p^{q+2}(\Omega)$ and $\frac{d^k p}{dt^k}  \in \mathcal{L}^{q-2k}$
for an even integer $q\ge 0$ with $k=0,\cdots,q/2$, we have $\frac{d^ku}{dt^k}\in \mathcal {L}^{q+4-2k}$ for $k=0,\cdots,q/2$. 
\end{remark}
Now we are ready to give the proof of Theorem \ref{theo:LWP}.
\begin{proof}[Proof of Theorem \ref{theo:LWP}]
Define the spaces and sets as follows
\begin{align*}
    V_0 &= \left\{ u\in \mathcal{L}^{q+4}: u(\cdot,0) \in H_p^{q+2}(\Omega),\frac{d^k u}{dt^k}\in\mathcal{L}^{q+4-2k},k=0,\cdots,q/2\right\}, \\
		V_1 &= \left\{u\in \mathcal{L}^{q+2}: u(\cdot,0) \in H_p^{q+1}(\Omega),\frac{d^k u}{dt^k}\in\mathcal{L}^{q+2-2k},k=0,\cdots,q/2\right\},\\
  V_2 &= \left\{p\in\mathcal{L}^q: \frac{d^k p}{dt^k}\in\mathcal{L}^{q-2k},k=0,\cdots,q/2 \right\},
\end{align*}
	and 
	\begin{align*}
 			U_0 &= \left\{\overrightarrow{u}:=(u_0,u_1,u_2,u_3) \in V_0 \times \left(V_1\right)^3 \right\}, \\
	U_1 &= \left\{\overrightarrow{\varphi}:=(\varphi_0,\varphi_1,\varphi_2,\varphi_3)\in H_p^{q+2}(\Omega) \times \left(H_p^{q+1}(\Omega)\right)^3\right\},\\
  U_2 &=  \left\{\overrightarrow{p}:=(p_0,p_1,p_2,p_3)\in\left(V_2\right)^4\right\},
	\end{align*}
	where  even number $q\ge d/2$.
We also define the error operator as
\begin{align*}
	&E : U_0 \times U_1 \to U_2\times U_1, 
	\end{align*}
	with
	\begin{align*}
E(\overrightarrow{u};\overrightarrow{\varphi}) = [E_{0},E_{1},E_{2},E_{3};I_{0},I_{1},I_{2},I_{3}],
\end{align*}
where $\overrightarrow{u}\in U_0$ and $\overrightarrow{\varphi}\in U_1$, and each component is defined as follows
\begin{align*}
	\begin{cases}
		E_0 = \partial_t u_0 + c_1\Delta^2 u_0 - \Delta \left(f_0(x,t,\overrightarrow{u})+g(x,t,u_0)\right), \\
		E_{i} = \partial_t u_i-c_2\Delta u_i-f_i(x,t,\overrightarrow{u}), \\
		I_0 = u_0(\cdot,0) - \varphi_0(\cdot), \\
		I_{i} = u_i(\cdot,0)-\varphi_i(\cdot),
	\end{cases}
\end{align*}
for $i=1,2,3$. To show that the error mapping $E$ is well-defined, we need to demonstrate that $\Delta f_0(x,t,\overrightarrow{u})$, $\Delta g(x,t,u_0)$, and $f_i(x,t,\overrightarrow{u})$ are in $V_2$. Let's take $\Delta g(x,t,u_0)$ and $k=0$ as an example. We need to prove the following condition holds based on the admissible conditions \ref{def:admissible 1}
\begin{align}\label{gl}
	\sum\limits_{\ell=1}^\infty \frac{1}{\ell!} \Delta \left(g^{(\ell)}(x,t)u_0^\ell \right)\in \mathcal{L}^q.
\end{align}	
This can be proved by expanding $\Delta \left(g^{(\ell)}(x,t)u_0^\ell \right)$ as
\begin{align*}
	&\Delta \left(g^{(\ell)}(x,t)u_0^\ell(x,t) \right)\\
	=&u_0^\ell(x,t)\Delta g^{(\ell)}(x,t)+g^{(\ell)}(x,t)\Delta u_0^\ell(x,t)+2\nabla g^{(\ell)}(x,t)\cdot \nabla u_0^{\ell}(x,t),
\end{align*}
which leads to the following inequality
\begin{align*}
	&\|\Delta \left(g^{(\ell)}(x,t)u_0^\ell(x,t) \right)\|_{\mathcal{L}^{q}}\\
	\leq &2\|g^{(\ell)}(x,t)\|_{\mathcal{L}^{q+2}}\left( \|u_0^\ell(x,t)\|_{\mathcal{L}^{q}}+\|\Delta u_0^\ell(x,t)\|_{\mathcal{L}^{q}}+\|\nabla u_0^\ell(x,t)\|_{\mathcal{L}^{q}}\right)\\
	\leq &2\|g^{(\ell)}(x,t)\|_{\mathcal{L}^{q+2}}\|u_0^\ell\|_{\mathcal{L}^{q+2}}\\
	\leq &2\|g^{(\ell)}(x,t)\|_{\mathcal{L}^{q+2}}\|u_0\|^\ell_{\mathcal{L}^{q+2}}.
\end{align*}
Since $u_0\in \mathcal{L}^{q+4}$, which implies that $\|u_0\|_{\mathcal{L}^{q+2}}$ is bounded, we have
\begin{align*}
		\left\|\sum\limits_{\ell=1}^\infty \frac{1}{\ell!} \Delta \left(g^{(\ell)}(x,t)u_0^\ell \right)\right\|_{\mathcal{L}^{q}}
		\leq 2\sum\limits_{\ell=1}^{\infty}\left\| g^{(\ell)}(x,t)\right\|_{\mathcal{L}^{q+2}}\frac{\|u_0\|^\ell_{\mathcal{L}^{q+2}}}{\ell!}
		\leq  C\exp{\left(\|u_0\|_{\mathcal{L}^{q+2}}\right)},
\end{align*}
where $C$ is a constant depending only on the uniform upper bound of $\left\| g^{(\ell)}(x,t)\right\|_{\mathcal{L}^{q+2}}$. Hence the condition \eqref{gl} holds. The proof for $\Delta f_0(x,t,\overrightarrow{u})$ and $f_i(x,t,\overrightarrow{u})$ follows a similar argument.

In order to show the local well-posedness of equation \eqref{equ:nonlinear system}, we will apply the implicit function theorem to the error operator  $E$. Before that it is necessary to prove  $E$ is analytic. By Lemma \ref{lemma:analytic}, one has to show  $E$ is locally bounded and weakly analytic. Specifically, the nonlinear part of $E$ must be weakly analytic. We also take $g(x,t,u_0)$ as an example. Due to the admissible condition for $g(x,t,\cdot)$, for any $u_{0,1}, u_{0,2} \in V_0$, let us consider the map
\begin{align*}
\lambda \mapsto \Delta \left(g(x,t,u_{0,1}+ \lambda u_{0,2}) \right).
\end{align*}
It holds
\begin{align*}
g(x,t,u_{0,1}+ \lambda u_{0,2})= \sum\limits_{\ell=1}^\infty \frac{g^{(\ell)}(x,t)(u_{0,1}+ \lambda u_{0,2})^\ell}{\ell!},
\end{align*}
which can be written as a polynomial of $\lambda$. Because of the boundedness of $u_{0,1}$, $u_{0,2}$, and $g^{(\ell)}$, it implies that there is a neighborhood of the origin in the complex plane that makes the series convergent, which proves that $E$ is analytic.

Recall that  $E(\overrightarrow{0};\overrightarrow{0})=0$ according to the Remark \ref{rmk:admissible condition}.  We proceed to calculate the Fr\'{e}chet derivative of $E$ with respect to  $\overrightarrow{u}$ at $\overrightarrow{0}$. Letting $\overrightarrow{h}=(h_0,h_1,h_2,h_3)\in U_0$, we have
  \begin{align*}    E(\overrightarrow{h};\overrightarrow{0})-E(\overrightarrow{0};\overrightarrow{0})
    =&\left[\partial_t h_0+c_1\Delta^2 h_0-\Delta \left(g^{(1)}(x,t)h_0+c_{1000}(x,t)h_0\right)+N_0(\overrightarrow{h})\right.,\\
    &\left.\partial_th_1-c_2\Delta h_1-b_1(x,t)h_1 +N_1(\overrightarrow{h})\right.,\\
    &\left.\partial_th_2-c_2\Delta h_2-b_2(x,t)h_2+N_2(\overrightarrow{h})\right.,\\
    &\left.\partial_th_3-c_2\Delta h_3-b_3(x,t)h_3+N_3(\overrightarrow{h})\right.,\\
&\left.h_0(x,0),h_1(x,0),h_2(x,0),h_3(x,0)\right]\\=&[A_0(h_0),A_1(h_1),A_2(h_2),A_3(h_3),\overrightarrow{h}(x,0)]\\&+[N_0(\overrightarrow{h}),N_1(\overrightarrow{h}),N_2(\overrightarrow{h}),N_3(\overrightarrow{h}),\overrightarrow{0}],
    \end{align*}
where  $A_0: V_0\mapsto V_2$ and $A_i: V_1\mapsto V_2$, $i=1,2,3,$ are the linear maps of  $\overrightarrow{h}$ and each component of $[N_0,N_1,N_2,N_3]:U_0\mapsto U_2$ is nonlinear, defined as 
\begin{align*}
&N_0(\overrightarrow{h}) = -\sum\limits_{s=2}^{\infty}\frac{1}{s!}\sum\limits_{l_0+l_1+l_2+l_3=s}c_{l_0l_1l_2l_3}(x,t)h_0^{l_0}h_1^{l_1}h_2^{l_2}h_3^{l_3}-\sum\limits_{l=2}^{\infty}\frac{g^{(l)}(x,t)}{l!}h_0^l,\\
&N_i(\overrightarrow{h})=-\sum\limits_{s=2}^{\infty} \frac{1}{s!}\sum\limits_{l_0+l_1+l_2+l_3=s} b_{i,l_0l_1l_2l_3}h_0^{l_0}h_1^{l_1}h_2^{l_2}h_3^{l_3},\,\,\,\text{for }i=1,2,3.
\end{align*}

It holds
$$\lim_{\|\overrightarrow{h}\|\to 0}\frac{\|N_i(\overrightarrow{h})\|}{\|\overrightarrow{h}\|}=0,$$ 
for $i=0,1,2,3$. Then the Fr\'{e}chet derivative of $E$ with respect to  $\overrightarrow{u}$ at $\overrightarrow{0}$, denoted as $E_{\overrightarrow{u}}(\overrightarrow{0})$, is expressed as \begin{align*}
    E_{\overrightarrow{u}}(\overrightarrow{0})\overrightarrow{w} = [A_0(w_0),A_1(w_1),A_2(w_2),A_3(w_3),w_0(x,0),w_1(x,0),w_2(x,0),w_3(x,0)],
\end{align*}
where $\overrightarrow{w}=(w_0,w_1,w_2,w_3)\in U_0$.
The  Fr\'{e}chet derivative $E_{\overrightarrow{u}}(\overrightarrow{0})$ is an isomorphism if and only if the following two equations are globally well-posed
\begin{equation*}
    \left\{
    \begin{aligned}
        &\partial_t u+c_1\Delta^2 u-\Delta \left[g^{(1)}(x,t)u+c_{1000}(x,t)u\right]=p_1(x,t),\ &&\text{in}\ Q,\\
        &u(x,0)=\varphi(x), &&\text{in}\ \Omega,\\
        &u(x)\ \text{satisfies PBC,}\ &&\text{on}\ \Sigma,
    \end{aligned}
    \right.
\end{equation*}
and 
\begin{equation*}
    \left\{
    \begin{aligned}
        &\partial_tv-c_2\Delta v-b_1(x,t)v=p_2(x,t)\ &&\text{in}\ Q,\\
        &v(x,0)=\psi(x) &&\text{in}\ \Omega,\\
        &v(x)\ \text{satisfies PBC,}\ &&\text{on}\ \Sigma,
    \end{aligned}
    \right.
\end{equation*}
where $\varphi\in H_p^{q+2}(\Omega),\psi\in H_p^{q+1}(\Omega), p_1,p_2\in V_2$ and $g^{(1)},c_{1000},b_1$ are given in the admissible conditions \ref{def:admissible 1}. Their well-posedness follows from Lemmas \ref{lemma:WP of Linear parabolic equation}, \ref{lemma: WP of Linear CH equation}, \ref{lemma:linear CH inprove regularity}, and Remark \ref{remark:higher regularity}. Hence, we obtain that $E_{\overrightarrow{u}}(\overrightarrow{0})$ is an isomorphism. 

Combining the fact that  $E(\overrightarrow{0};\overrightarrow{0})=0$, and making use of the implicit function theorem given by Lemma \ref{lemma:implicit function},  we conclude that if
\begin{align*}
    \mathcal{N}_{\delta}=\left\{\overrightarrow{\varphi}\in U_1;\|\varphi_0\|_{H_p^{q+2}(\Omega)}+\sum\limits_{i=1}^3\|\varphi_i\|_{H_p^{q+1}(\Omega)}<\delta\right\},
\end{align*}
for a small $\delta>0$, then there exists function $\overrightarrow{F}:\mathcal{N}_{\delta}\to U_0$ satisfying $E(\overrightarrow{\varphi};\overrightarrow{F}(\overrightarrow{\varphi}))=0$, which implies the existence to equation (\ref{equ:nonlinear system}). Finally, due to the local Lipschitz property of $\overrightarrow{F}$, we have

\begin{align*}
\sum\limits_{k=0}^{q/2}\left(\left\|\frac{d^ku_0}{dt^k}\right\|_{\mathcal{L}^{q+4-2k}}+\sum\limits_{i=1}^3\left\|\frac{d^ku_i}{dt^k}\right\|_{\mathcal{L}^{q+2-2k}}\right)\leq C\delta,
\end{align*}
where $C$ is a constant depending on $T$ and $\Omega$, which ends the proof.
\end{proof}

\section{Uniqueness proof on the inverse problems}
\label{section:Proof of inverse problem}
In this section, we will employ the higher order linearization method~\cite{KLU14} to deal with the nonlinearity and coupling properties of the phase-field system \eqref{equ:nonlinear system}, and subsequently provide uniqueness proofs for inverse problems (IP1) and (IP2).
\subsection{Higher order linearization}

In the proof of Theorem \ref{theo:LWP}, we obtain the smooth dependence of the solution $\overrightarrow{u}$ on the initial value $\overrightarrow{\varphi}$. Therefore, for any $n\ge 1$, assuming the initial value $\overrightarrow{\varphi}$ is a polynomial of a small parameter $\epsilon$ of degree $n$,  $\overrightarrow{u}$ must be smooth with respect to $\epsilon$. More specifically, let us assume that the initial value of equation (\ref{equ:nonlinear system}) can be expanded as follows

\begin{align}
\label{equ:initial_value_expansion}
	\begin{cases}
    u_0(x,0)=\varphi_0(x)=\epsilon\varphi_0^{(1)}(x),\\
    u_i(x,0)=0,\ i=1,2,3.
    \end{cases}
\end{align}

\begin{remark}
 In contrast to the standard higher order linearization approach, which is formulated as
\begin{align*}
u_0(x,0) = \varphi_0 = \sum\limits_{j=1}^n\varepsilon_j\varphi_0^{(j)}(x), \end{align*} our method can be viewed as a special case with $n=1$. Given that we have access to all interior measurements at a given time, this formulation is sufficient for application in the following inverse problems. Additionally, it is also more efficient in handling constraints on the initial concentration field, such as non-negativity constraints.
\end{remark}
In order to ensure the well-posedness of equation (\ref{equ:nonlinear system}) as stated in Theorem \ref{theo:LWP}, we choose that

\begin{align}
\label{equ:condtion_initial_velue}
\epsilon\|\varphi_0^{(1)}\|_{H_p^{q+2}(\Omega)}\leq \delta,\,\,\, \text{with }\varphi_0^{(1)}\not\equiv 0\text{ in $\Omega$},
\end{align}
where $\delta>0$ is a small parameter introduced in Theorem \ref{theo:LWP}.

Now, let us express the solution of equation (\ref{equ:nonlinear system}) as a function of $\epsilon$:
\begin{align*}
	\begin{cases}
    u_0(\epsilon) = u_0(x,t,\epsilon),\\
    u_i(\epsilon) = u_i(x,t,\epsilon), \ i=1,2,3,\\
    \overrightarrow{u}(\epsilon) = [u_0(\epsilon),u_1(\epsilon),u_2(\epsilon),u_3(\epsilon)],
    \end{cases}
\end{align*}
with $\overrightarrow{u}(0) = \overrightarrow{0}$ and define
\begin{align*}
	\begin{cases}
    u_0^{(1)}:=\frac{\partial u_0}{\partial \epsilon}\big|_{\epsilon=0},\\
 u_i^{(1)}:=\frac{\partial u_i}{\partial \epsilon}\big|_{\epsilon=0}, \ i=1,2,3, 
  \end{cases}
\end{align*}
where $u_0^{(1)}$ and $u_i^{(1)}$ satisfy the PBC on  $\Sigma$. 
For the time derivative $\partial_t u_i$, we have 
\begin{align}
\label{equ:time derivative}
    \frac{\partial (\partial_t u_i)}{\partial \epsilon}\bigg|_{\epsilon=0}=\lim\limits_{\epsilon\to 0} \frac{\partial_t \left(u_i(x,t,\epsilon)-u_i(x,t,0)\right)}{\epsilon}=\partial_t \lim\limits_{\epsilon\to 0}\frac{u_i(x,t,\epsilon)}{\epsilon} = \partial_t u_i^{(1)}
\end{align}
Meanwhile, for the nonlinear function $f_1$, it holds
\begin{equation}
    \label{equ:nonlinear derivative}
    \begin{aligned}
    &\frac{\partial f_1(x,t,\overrightarrow{u})}{\partial \epsilon}\bigg|_{\epsilon=0}\\=&\lim\limits_{\epsilon\to 0}\frac{f_1\left(x,t,\overrightarrow{u}(\epsilon)\right)-f_1\left(x,t,\overrightarrow{u}(0)\right)}{\epsilon}\\
    =&b_1(x,t)u_1^{(1)}+\lim\limits_{\epsilon\to 0}\frac{1}{\epsilon}\sum\limits_{s=2}^{\infty} \frac{1}{s!}\sum\limits_{l_0+l_1+l_2+l_3=s} b_{1,l_0l_1l_2l_3}(x,t)u_0^{l_0}(\epsilon)u_1^{l_1}(\epsilon)u_2^{l_2}(\epsilon)u_3^{l_3}(\epsilon)\\
    =&b_1(x,t)u_1^{(1)}.
\end{aligned}
\end{equation}
One can analogously find  the derivatives of $f_2$ and $f_3$ with respect to $\epsilon$. We then differentiate both sides of equation (\ref{equ:nonlinear system}) with respect to $\epsilon$ and let $\epsilon=0$. It yields that $u_1^{(1)}$ satisfies the following equation:
\begin{equation}
\label{equ:linear_parabolic_eqution}
    \left\{
    \begin{aligned}
        &\partial_tu_1^{(1)}=c_2\Delta u_1^{(1)}+b_1(x,t)u_1^{(1)},&&\text{in}\ Q,\\
        &u^{(1)}_1(x,0)=0,&&\text{in}\ \Omega,\\
        &u_1^{(1)}(x,t)\ \text{satisties PBC,}&&\text{on}\ M.
    \end{aligned}
    \right.
\end{equation}
Note that $u^{(1)}_1(x,t)\equiv 0$ is a unique solution to equation \eqref{equ:linear_parabolic_eqution} by Lemma \ref{lemma:WP of Linear parabolic equation}. Similarly, it also holds that $u^{(1)}_2(x,t)=u^{(1)}_3(x,t)\equiv 0$.

For the nonlinear functions $f_0$ and $g$, we follow the same approach to find 
\begin{align*}
    \frac{\partial f_0(x,t,\overrightarrow{u})}{\partial \epsilon}\bigg|_{\epsilon=0}=&\lim\limits_{\epsilon\to 0}\frac{1}{\epsilon}\sum\limits_{s=1}^{\infty} \frac{1}{s!}\sum\limits_{l_0+l_1+l_2+l_3=s} c_{l_0l_1l_2l_3}(x,t)u_0^{l_0}(\epsilon)u_1^{l_1}(\epsilon)u_2^{l_2}(\epsilon)u_3^{l_3}(\epsilon)\\
    =&\sum\limits_{s=1}^{\infty}\frac{1}{s!}c_{s000}(x,t)\lim\limits_{\epsilon\to 0}\frac{u_0^s(\epsilon)}{\epsilon}\\
    =&c_{1000}u_0^{(1)},
    \end{align*}
    and
    \begin{align*}
\frac{\partial g_0(x,t,u_0(\epsilon))}{\partial \epsilon}\bigg|_{\epsilon=0}
=&\sum\limits_{k=1}^{\infty}\frac{1}{k!}g^{(\ell)}(x,t)\lim\limits_{\epsilon\to 0}\frac{u_0^{k}(\epsilon)}{\epsilon}\\
=&g^{(1)}(x,t)u_0^{(1)}.
\end{align*}
This leads to the equation for $u_0^{(1)}$:
\begin{equation}
\label{equ:first order linearization}
    \left\{
    \begin{aligned}
        &\partial_tu_0^{(1)}=-c_1\Delta^2 u_0^{(1)}+\Delta\left(c_{1000}(x,t)u_0^{(1)}+g^{(1)}(x,t)u_0^{(1)}\right) ,&&\text{in}\ Q,\\
        &u^{(1)}_0(x,0)=\varphi^{(1)}_0,&&\text{in}\ \Omega,\\
        &u_0^{(1)}(x,t)\ \text{satisties PBC,}&&\text{on}\ \Sigma.
    \end{aligned}
    \right.
\end{equation}
For the second order of $\overrightarrow{u}$, we can define 
\begin{align*}
	\begin{cases}
    u_0^{(2)}:=\frac{\partial^2 u_0}{\partial \epsilon^2}\big|_{\epsilon=0},\\
 u_i^{(2)}:=\frac{\partial^2 u_i}{\partial \epsilon^2}\big|_{\epsilon=0}, \ i=1,2,3, \\
  \end{cases}
\end{align*}
where $u_0^{(2)}$ and $u_i^{(2)}$ satisfy PBC on $\Sigma$. Take the second order derivative with respect to $\epsilon$ for equation (\ref{equ:nonlinear system}). Similar to equations \eqref{equ:time derivative} and \eqref{equ:nonlinear derivative}, we end up with equation
\begin{equation*}
     \left\{
    \begin{aligned}
        &\partial_t u_1^{(2)}=c_2\Delta u_1^{(2)}+b_1(x,t)u_1^{(2)},\ &&\text{in}\ Q,\\
        &u^{(2)}_1(x,0)=0,\ &&\text{in}\ \Omega,\\
        &u_1^{(2)}(x,t)\ \text{satisties PBC,}&&\text{on}\ \Sigma.
    \end{aligned}
    \right.
\end{equation*}
It holds $u_1^{(2)}(x,t)=u_2^{(2)}(x,t)=u_3^{(2)}(x,t)\equiv 0$ in $Q$ and $u_0^{(2)}$ satisfies
\begin{equation}
\label{equ:second order linearization}
    \left\{
    \begin{aligned}
        &\partial_tu_0^{(2)}=-c_1\Delta^2 u_0^{(2)}\\
        &+\Delta\left[\left(g^{(1)}(x,t)+c_{1000}(x,t)\right)u_0^{(2)}+\left(g^{(2)}(x,t)
        +c_{2000}(x,t)\right)\left(u_0^{(1)}\right)^2
        \right],&&\text{in}\ Q,\\
        &u^{(2)}_0(x,0)=0,&&\text{in}\ \Omega,\\
        &u_0^{(2)}(x,t)\ \text{satisties PBC,}&&\text{on}\ \Sigma.\\
    \end{aligned}
    \right.
\end{equation}
In general, we can consider the $\ell$th-order derivative with respect to $\epsilon$ for $1\leq \ell \leq n$,
\begin{align*}
	\begin{cases}
    u_0^{(\ell)}:=\frac{\partial^{\ell} u_0}{\partial \epsilon^{\ell}}\big|_{\epsilon=0},\\
 u_i^{(\ell)}:=\frac{\partial^{\ell} u_i}{\partial \epsilon^{\ell}}\big|_{\epsilon=0}, \ i=1,2,3,
  \end{cases}
\end{align*}
where $u_0^{(\ell)}$ and $u_i^{(\ell)}$ satisfy PBC on $\Sigma$. We assert that $u_i^{(\ell)}\equiv 0$, $i=1,2,3$, in $Q$ by the same argument as before. Following the same calculations as \eqref{equ:time derivative} and \eqref{equ:nonlinear derivative}, we have that $u_0^{(\ell)}$ satisfies the equation:
\begin{equation}
\label{equ:ell order linearization}
    \left\{
    \begin{aligned}
        &\partial_t u_0^{(\ell)} = -c_1\Delta^2 u_0^{(\ell)}\\
        &+\Delta \left(\sum\limits_{i_1+2i_2+\cdots+(\ell-1)i_{\ell-1}=\ell}d_{\ell}c_{m_{\ell}000}(x,t)\left(u_0^{(1)}\right)^{i_1}\cdots \left(u_0^{(\ell-1)}\right)^{i_{\ell-1}}\right)\\
        &+\Delta\left( \sum\limits_{i_1+2i_2+\cdots+(\ell-1)i_{\ell-1}=\ell}d_{\ell}g^{(m_{\ell})}(x,t)\left(u_0^{(1)}\right)^{i_1}\cdots \left(u_0^{(\ell-1)}\right)^{i_{\ell-1}}\right), &&\text{in}\ Q,\\
        &u_0^{(\ell)}(x,0)=0,&&\text{in}\ \Omega,\\
        &u_0^{(\ell)}(x,t)\ \text{satisties PBC,}&&\text{on}\ \Sigma.
    \end{aligned}
    \right.
\end{equation}
where $m_{\ell} = \sum\limits_{j=1}^{\ell-1}i_j$ and $d_{\ell}$ is a coefficient independent of $x$ and $t$.

\subsection{Proof of uniqueness on (IP1)}
 We are now ready to prove Theorem \ref{theo:IP1} for the  unique determination of $g$ that is independent of time. 
In order to establish the uniqueness, our main tool is the higher order linearization method introduced in the last subsection. This involves applying higher order linearization to the initial values, and using integration by parts to establish a connection between the measurements at a given time and the nonlinear coefficients  $g^{(\ell)}$. Details are given in the following.

Firstly, under the conditions of (IP1), $g^{(\ell)}$ depends solely on $x$ for all $\ell$. Let us consider the measurement  $\overrightarrow{M}_s(g,\overrightarrow{\varphi})$ given by
\begin{align*}
\overrightarrow{M}_s(g,\overrightarrow{\varphi}) = [u_0(x,t_1,\epsilon),\partial_t u_0(x,t_1,\epsilon)],
\end{align*}
for some $t_1\in (0,T]$ with initial value $\overrightarrow{\varphi}(\epsilon)$ given by equation \eqref{equ:initial_value_expansion}, so that the $\ell$th-order linearization of $\overrightarrow{M}_s(g,\overrightarrow{\varphi})$ is
\begin{align*}
    \overrightarrow{M}^{(\ell)}_s(g,\overrightarrow{\varphi})=\partial^{\ell}_{\epsilon}\overrightarrow{M}_s(g,\overrightarrow{\varphi})\big|_{\epsilon=0} &= \left[\partial^{\ell}_{\epsilon}u_0(x,t_1,\epsilon)\big|_{\epsilon=0} ,\partial^{\ell}_{\epsilon}\partial_t u_0(x,t_1,\epsilon)\big|_{\epsilon=0}\right]\\&=\left[u_0^{(\ell)}(x,t_1),\partial_t u_0^{(\ell)}(x,t_1)\right],
\end{align*}
for $1\leq \ell\leq n$.

To prove that we can uniquely determine $g^{(1)}(x)$ with the measurement $\overrightarrow{M}_s^{(1)}$, we denote $u_{j,0}$, the solution for system \eqref{equ:nonlinear system} with energy $g_j$, $j=1,2$. One can make use of equation \eqref{equ:first order linearization} to derive the equation for the first order linearization $u_{j,0}^{(1)}$ of $u_{j,0}$:
\begin{equation}
\label{equ:u_j^1}
    \left\{
    \begin{aligned}
        &\partial_t u_{j,0}^{(1)}=-c_1\Delta^2 u_{j,0}^{(1)}+\Delta \left(g_j^{(1)}(x)u_{j,0}^{(1)}+c_{1000}(x,t)u_{j,0}^{(1)}\right),&&\text{in}\ Q,\\
        &u_{j,0}^{(1)}(x,0)=\varphi_0^{(1)}(x),&&\text{in}\ \Omega,\\
         &u_{j,0}^{(1)}(x,t)\ \text{satisties PBC,}&&\text{on}\ \Sigma.
    \end{aligned}
    \right.
\end{equation}
 The two systems with $j=1, 2$ have the same measurement, i.e.
\begin{align*}
   u_{1,0}^{(1)}(x,t_1)=u_{2,0}^{(1)}(x,t_1),\ \partial_t u_{1,0}^{(1)}(x,t_1)=\partial_t u_{2,0}^{(1)}(x,t_1).
\end{align*}
Let $v^{(1)}=u_{1,0}^{(1)}-u_{2,0}^{(1)}$. Then it holds
\begin{align}
\label{equ:property u1}
    v^{(1)}(x,0)=v^{(1)}(x,t_1)=\partial_t v^{(1)}(x,t_1)=0,\ \text{for all}\ x\in\Omega.
\end{align}

By subtracting equation (\ref{equ:u_j^1}) with $j=1$ from the same equation with $j=2$ and multiplying both sides by any function $w\in H_p^2(\Omega)$, we then integrate over $\Omega$ to obtain
\begin{align*}
   \int_{\Omega}w\partial_t v^{(1)}\mathrm{d}x =& \int_{\Omega}-c_1w\Delta^2 v^{(1)}+w\Delta\left(g_1^{(1)}(x)v^{(1)}\right) \\&+w\Delta\left(g_1^{(1)}(x)u_{2,0}^{(1)}-g_{2}^{(1)}(x) u_{2,0}^{(1)}\right)\mathrm{d}x.
\end{align*}
Substituting $t=t_1$ in the integral and using the condition  (\ref{equ:property u1}) lead to
\begin{align*}
    \int_{\Omega}w\Delta\left[\left(g_1^{(1)}(x)-g_2^{(1)}(x)\right)u_{2,0}^{(1)}(x,t_1)\right]\mathrm{d}x=0.
\end{align*}
To simplify the notations,  let $u_0^{(1)}(x,t_1)=u_{1,0}^{(1)}(x,t_1)=u_{2,0}^{(1)}(x,t_1)$. Applying integration by parts, we obtain
\begin{align*}
    \int_{\Omega}\nabla w\cdot \nabla \left[\left(g_1^{(1)}(x)-g_2^{(1)}(x)\right)u_{0}^{(1)}(x,t_1)\right]\mathrm{d}x=0.
\end{align*}
By setting \( w = \left(g_1^{(1)}(x) - g_2^{(1)}(x)\right) u_{0}^{(1)}(x,t_1) \), we observe that \( \left(g_1^{(1)}(x) - g_2^{(1)}(x)\right) u_{0}^{(1)}(x,t_1) \) is constant almost everywhere in \(\Omega\). However, due to the arbitrariness of the initial condition \(\varphi_0^{(1)}(x)\), $u_{0}^{(1)}(x,t_1)$ is also an arbitrarily varying  function from equation \eqref{equ:u_j^1}.  This directly implies that \(g_1^{(1)}(x) = g_2^{(1)}(x)\) almost everywhere in \(\Omega\). A discussion on inverting \(g_1^{(1)}(x)\) in a special case can be found in Remark \ref{remark41}.

Let us denote $g^{(1)}(x) = g_1^{(1)}(x)=g_2^{(1)}(x)$. Analogous to the proof of determining the first order coefficient, using \eqref{equ:second order linearization} one can show the second order linearization $u^{(2)}_{j,0}$ satisfies
\begin{equation}
\label{equ:u_j^2}
    \left\{
    \begin{aligned}
        &\partial_tu_{j,0}^{(2)}=-c_1\Delta^2 u_{j,0}^{(2)}\\
        &+\Delta\left(\left(g^{(1)}(x)+c_{1000}(x,t)\right)u_{j,0}^{(2)}+\left(g_j^{(2)}(x)
        +c_{2000}(x,t)\right)\left(u_0^{(1)}\right)^2
        \right),&&\text{in}\ Q,\\
        &u_{j,0}^{(2)}(x,0)=0,&&\text{in}\ \Omega,\\
        &u_{j,0}^{(2)}(x,t)\ \text{satisties PBC,}&&\text{on}\ \Sigma.
    \end{aligned}
    \right.
\end{equation}
for $j=1,2$. The two systems have the same measurement, that is, 
\begin{align*}
   u_{1,0}^{(2)}(x,t_1)=u_{2,0}^{(2)}(x,t_1),\ \partial_t u_{1,0}(x,t_1) =\partial_t u_{2,0}(x,t_1).
\end{align*}
Let $v^{(2)}=u_{1,0}^{(2)}-u_{2,0}^{(2)}$. Then we have
\begin{align}
\label{equ:property u}
    v^{(2)}(x,0)=v^{(2)}(x,t_1)=\partial_t v^{(2)}(x,t_1)=0,\ \text{for all}\ x\in\Omega.
\end{align}
By subtracting equation \eqref{equ:u_j^2} with $j=1$ from the same equation with $j=2$, and multiplying both sides by a test function $w\in H_p^2(\Omega)$, we integrate over $\Omega$ and substitute $t=t_1$ to obtain:
\begin{align}
    \int_{\Omega}w\Delta \left[\left(g_1^{(2)}(x)-g_2^{(2)}(x)\right)\left(u_0^{(1)}(x,t_1)\right)^2\right]\mathrm{d}x=0.
\end{align}
Following the same argument, we conclude that
 \begin{align}
 \label{equ:tilde and hat a}
     \left(g_1^{(2)}(x)-g_2^{(2)}(x)\right)\left(u_0^{(1)}(x,t_1)\right)^2\text{ is constant a.e. in } \Omega,
 \end{align}
which implies $g_1^{(2)}(x)=g_2^{(2)}(x)$ in $\Omega$ due to the arbitrariness of $u_0^{(1)}(x,t_1)$. We thus denote $g^{(2)}=g_1^{(2)}=g_2^{(2)}$ in $\Omega$.

We make use of mathematical induction to prove the uniqueness of inversion for $g^{(\ell)}, \ell> 2$. Assume that $g^{(1)}, \ldots, g^{(\ell-1)}$ have been uniquely determined through measurements $\overrightarrow{M}_s(g_j,\overrightarrow{\varphi})$, $j=1,2$. To determine $g^{(\ell)}$, we consider $u_{1,0}^{(\ell)}$ and $u_{2,0}^{(\ell)}$ as the solutions associated with the coefficients $g_1^{(\ell)}$ and $g_2^{(\ell)}$, respectively. They satisfy the following systems as derived from equation \eqref{equ:ell order linearization}:
\begin{equation*}
\label{equ:hat uk}
    \left\{
    \begin{aligned}
        &\partial_t u_{j,0}^{(\ell)} = -c_1\Delta^2 u_{j,0}^{(\ell)}\\
        &+\Delta \left[\left(g^{(1)}(x)+c_{1000}(x,t)\right)u_{j,0}^{(\ell)}+ \left(g_j^{(\ell)}(x)+c_{\ell000}(x,t)\right)\left(u_0^{(1)}\right)^{\ell}\right]\\
        &+\Delta \left(\sum_{\stackrel{i_1+2i_2+\cdots+(\ell-1)i_{\ell-1}=\ell}{i_1\neq \ell}}d_{\ell}c_{m_{\ell}000}(x,t)u_0^{(1),i_1}\cdots u_0^{(\ell-1),i_{\ell-1}}\right)\\
        &+\Delta\left( \sum_{\stackrel{i_1+2i_2+\cdots+(\ell-1)i_{\ell-1}=\ell}{i_1\neq \ell}}d_{\ell}g^{(m_{\ell})}(x)u_0^{(1),i_1}\cdots u_0^{(\ell-1),i_{\ell-1}}\right), &&\text{in}\ Q,\\
        &u_0^{(\ell)}(x,0)=0,&&\text{in}\ \Omega,\\
        &u_{j,0}^{(\ell)}(x,t)\ \text{satisfies PBC,}&&\text{on}\ \Sigma.
    \end{aligned}
    \right.
\end{equation*}
where $m_{\ell} =\sum\limits_{j=1}^{\ell-1}i_j $, and $u_{1,0}^{(\ell)}$, $u_{2,0}^{(\ell)}$ satisfy the same measurement at $t=t_1$:
\begin{align*}
    u_{1,0}^{(\ell)}(x,t_1)=u_{2,0}^{(\ell)}(x,t_1),\partial_t u_{1,0}^{(\ell)}(x,t_1)=\partial_t u_{2,0}^{(\ell)}(x,t_1).
\end{align*}
Following the same argument, we obtain the identity
    \begin{align*}
         (g_{1,0}^{(\ell)}(x)-g^{(\ell)}_{2,0}(x))\left(u_0^{(1)}(x,t_1)\right)^{\ell}\equiv C\text{ a.e. in }\Omega,
    \end{align*}
which implies $g_{1,0}^{(\ell)}(x)=g^{(\ell)}_{2,0}(x)$  in $\Omega$. 

In conclusion, we have demonstrated the uniqueness of the coefficients $g^{(\ell)}$ within $\Omega$, where $1\leq \ell\leq n$ for arbitrary $n\ge 1$. This ends the proof of Theorem \ref{theo:IP1}. 
\begin{remark}\label{remark41}
For some phase-field models that have energy terms  independent of space and time, i.e., the coefficients $g^{(\ell)}, \ell \ge 1$ are constants, these coefficients can be explicitly determined. For example, the solution  for the coefficient $g^{(1)}$ can be found by
\begin{align*}
    g^{(1)} = \frac{\partial_t u_0^{(1)}(x_1,t_1)+c_1\Delta^2u_0^{(1)}(x_1,t_1)-\Delta (c_{1000}(x_1,t_1)u_0^{(1)}(x_1,t_1))}{\Delta u_0^{(1)}(x_1,t_1)},
\end{align*}
where $x_1 \in \Omega$ such that $\Delta u_0^{(1)}(x_1, t_1) \neq 0$. Moewover, if $f_0(x,t,\overrightarrow{u}) = f_0(\overrightarrow{u})$, we can explicitly invert $g^{(\ell)}$ by measuring  $u_0(x,t_1,\epsilon)$ only (i.e. without $\partial_t u_0(x,t_1,\epsilon)$). Let us consider $\ell = 1$ as an example. One can transform equation \eqref{equ:first order linearization} into 
\begin{align}
\label{equ:u_0^1}
\partial_tu_0^{(1)}=-c_1\Delta^2 u_0^{(1)}+\Delta\left(c_{1000}u_0^{(1)}+g^{(1)}u_0^{(1)}\right),
\end{align}
with initial value $\varphi^{(1)}_0$. Applying Fourier transform to both sides of equation \eqref{equ:u_0^1} yields
\begin{equation}
\label{equ:Fourier_tranform_u}
    \left\{
    \begin{aligned}
&\partial_t\mathcal{F}\left(u_0^{(1)}\right)=-\left(c_1|\xi|^4+c_{1000}|\xi|^2+g^{(1)}|\xi|^2\right) \mathcal{F}\left(u_0^{(1)}\right), \\
&\mathcal{F}\left(u_0^{(1)}\right)(\xi,0) = \mathcal{F}\left(\varphi^{(1)}_0\right)(\xi),
\end{aligned}
    \right.
\end{equation}
where $\mathcal{F}(\cdot)$ means the Fourier transform, and $\xi$ is frequency variables.  The solution to equation \eqref{equ:Fourier_tranform_u} is
\begin{align}
\label{equ:first order linearization solution}
    \mathcal{F}\left(u_0^{(1)}\right)(\xi,t) = \mathcal{F}\left(\varphi^{(1)}_0\right)(\xi)\exp\left[-\left(c_1|\xi|^4+c_{1000}|\xi|^2+g^{(1)}|\xi|^2\right) t\right].
\end{align}
Substituting $t = t_1$ in equation (39) yields
\begin{align*}
g^{(1)}=\ln\left(\frac{\mathcal{F}\left(\varphi^{(1)}_0\right)(\xi_1)}{\mathcal{F}\left(u_0^{(1)}\right)(\xi_1,t_1)\exp\left(c_1t_1|\xi_1|^4+c_{1000}t_1|\xi_1|^2\right)}\right)t_1^{-1}|\xi_1|^{-2},
\end{align*}
where $\xi_1\neq 0$ satisfies $\mathcal{F}\left(u_0^{(1)}\right)(\xi_1,t_1)\neq 0$. According to (39), we only need to set \(\varphi_0^{(1)}\) such that \(\mathcal{F}\left(\varphi^{(1)}_0\right)(\xi) \not\equiv 0\). For example, we can choose \(\varphi_0^{(1)} = 1+\cos(k \cdot x)\), where \(k\) is any non-zero vector in \(\mathbb{R}^d\). The other coefficients can be found by a similar fashion. 
\end{remark}

\subsection{Proof of uniqueness on (IP2)}
  In the  (IP2),  the coefficients $g^{(\ell)}$ depends on both $x$ and $t$ for all $\ell$. To determine the coefficients at time $t_j$ using the higher order linearization method, we follow a similar  approach as (IP1). Consider the measurement $\overrightarrow{M}_m(g,\overrightarrow{\varphi})$ given by
\begin{align*}
    \overrightarrow{M}_m(g,\overrightarrow{\varphi}) = [u_0(x,t_1,\epsilon),\partial_t u_0(x,t_1,\epsilon),\cdots,u_0(x,t_N,\epsilon),\partial_t u_0(x,t_N,\epsilon)],
\end{align*}
for $t_1,\cdots,t_N\in (0,T]$, so that the $\ell$th-order linearization of $\overrightarrow{M}_m(g,\overrightarrow{\varphi})$ is
\begin{align*}
    \overrightarrow{M}^{(\ell)}_m(g,\overrightarrow{\varphi})&=\partial^{\ell}_{\epsilon}\overrightarrow{M}_m(g,\overrightarrow{\varphi})\big|_{\epsilon=0} \\
    &= \left[\partial^{\ell}_{\epsilon}u_0(x,t_1,\epsilon)\big|_{\epsilon=0} ,\partial^{\ell}_{\epsilon}\partial_t u_0(x,t_1,\epsilon)\big|_{\epsilon=0},\cdots,\partial^{\ell}_{\epsilon}u_0(x,t_N,\epsilon)\big|_{\epsilon=0},\partial^{\ell}_{\epsilon}\partial_t u_0(x,t_N,\epsilon)\big|_{\epsilon=0}\right]\\
    &=\left[u_0^{(\ell)}(x,t_1),\partial_t u_0^{(\ell)}(x,t_1),\cdots,u_0^{(\ell)}(x,t_N),\partial_t u_0^{(\ell)}(x,t_N)\right],
\end{align*}
for $1\leq \ell\leq n$.


Suppose that the coefficients $g_{1}^{(1)}(x,t)$ and $g_{2}^{(1)}(x,t)$ determine the solutions $u_{1,0}^{(1)}$ and $u_{2,0}^{(1)}$, respectively, under the systems
\begin{equation*}
    \left\{
    \begin{aligned}
        &\partial_t u_{j,0}^{(1)}=-c_1\Delta^2 u_{j,0}^{(1)}+\Delta \left(g_j^{(1)}(x,t)u_{j,0}^{(1)}+c_{1000}(x,t)u_{j,0}^{(1)}\right),&&\text{in}\ Q,\\
        &u_{j,0}^{(1)}(x,0)=\varphi_0^{(1)}(x),&&\text{in}\ \Omega,\\
        &u_{j,0}^{(1)}(x,t)\ \text{satisties PBC,}&&\text{on}\ \Sigma.
    \end{aligned}
    \right.
\end{equation*}
for $j=1,2$. These two systems have the same measurement at $t=t_i$, that is, 
\begin{align*}
    u_{1,0}^{(1)}(x,t_i) = u_{2,0}^{(1)}(x,t_i),\,\,\,\,\, \partial_t u_{1,0}^{(1)}(x,t_i) = \partial_t u_{2,0}^{(1)}(x,t_i),\,\,\,\,\,i=1,\cdots,N.
\end{align*}
By employing the same procedure as in (IP1), we deduce that 
\begin{align*}
    g_1^{(1)}(x,t_i) = g_2^{(1)}(x,t_i)\text{ 
     in }\Omega,
\end{align*}
for $i=1,\cdots,N$. The uniqueness proof of $g^{(\ell)}, \ell\ge 2$, also follows the argument before, which gives 
\begin{align*}
    g_1^{(\ell)}(x,t_i) = g_2^{(\ell)}(x,t_i)\text{  in }\Omega,
\end{align*}
for $i=1,\cdots,N$. 

In the end, we can interpolate $g_1^{(\ell)}(x,t)$ and $g_2^{(\ell)}(x,t)$ with the $N$-th order polynomial with respect to time. Then based on the Cauchy remainder of polynomial interpolation \cite{AI48}, we deduce that
\begin{align*}
    \left\|g_1^{(\ell)}(x,t)-g_2^{(\ell)}(x,t)\right\|_{L^{\infty}(Q)}\leq \frac{C_{\ell}}{(N+1)!},\,\,\,\,\,\ (x,t)\in \Omega\times [0,T],
\end{align*}
where $C_{\ell}$ is determined by $\|\partial_t^{N+1}g_1^{(\ell)}\|_{L^{\infty}(Q)},\|\partial_t^{N+1}g_2^{(\ell)}\|_{L^{\infty}(Q)}$ and $T$.
In conclusion, Theorem \ref{theo:IP2} has been established.

\section{Conclusion}
In this work, we establish the existence and uniqueness of weak solutions for linear Cahn-Hilliard equations under the periodic boundary conditions. We also prove the local well-posedness of the Cahn-Hilliard-Allen-Cahn system. We then discuss two inverse problems related to the unique determination of nonlinear energy functions using different amount of measurement data, namely, the measurement of the concentration field at a fixed time point and the  measurements at multiple  time points within  the material. These measurements allow us to uniquely recover the nonlinear energy functions that are time-independent and time-dependent, respectively. Our method can be generalized to the inverse problems in phase-field models with other boundary conditions. 

However, there are still lots of challenges remaining. First, the inverse problems for simultaneously recovering all the energy terms in equation \eqref{equ:nonlinear system}, including $g$ and $f_i$, $i=0,1,2,3$, are still open. Second, it is difficult to measure the time derivative of the concentration field in general. Although we have discussed some special cases in remark \ref{remark41}, the uniqueness theory based on the field only (without time derivative) for general cases is still unknown. Third, the numerical implementation for recovering the energy terms is highly nontrivial, due to the nonlinearities in both forward and inverse problems. We leave these discussions for future work.

\appendix
\section{Proof on the uniqueness of system \eqref{equ:nonlinear system} under the zero initial condition}
Here we discuss the uniqueness of the solution for the phase-field system \eqref{equ:nonlinear system with zeros initial} under zero initial values. We only consider the uniqueness of solutions satisfying $u_0\in\mathcal{L}^{q+4}$ and $u_i\in \mathcal{L}^{q+2}$ for $i=1,2,3$, where $q\ge 0$ is given in Theorem \ref{theo:LWP}.
\begin{theorem}
	Let $\overrightarrow{u}=(u_0,u_1,u_2,u_3)\in \mathcal{L}^{q+4}\times \left(\mathcal{L}^{q+2}\right)^3$  satisfy the following equation
	\begin{equation}
		\label{equ:nonlinear system with zeros initial}
		\left\{
		\begin{aligned}
			&\partial_t u_0=-c_1\Delta^2u_0+\Delta \left(f_0(x,t,\overrightarrow{u})+g(x,t,u_0)\right),\ &&\text{in}\ Q,\\
			&\partial_t u_i = c_2\Delta u_i+f_i(x,t,\overrightarrow{u}),\ &&\text{in}\ Q,\\
			&u_0(x,0) = 0, &&\text{in}\ \Omega,\\
			&u_i(x,0) = 0, &&\text{in}\ \Omega,\\
			&u_0(x,t),u_i(x,t)\ \text{satisfy PBC}, &&\text{on}\ \Sigma,
		\end{aligned}
		\right.
	\end{equation}
	where $g$ and $f_i$, $i=0,1,2,3$, satisfy admissible conditions \ref{def:admissible 1}.  Then there exists a unique solution to equation \eqref{equ:nonlinear system with zeros initial}, which is $\overrightarrow{u} \equiv \overrightarrow{0}$.
	

\end{theorem}
\begin{remark}
Through condition 4 in the admissible conditions \ref{def:admissible 1}, we can derive the Lipschitz continuity of $f_i$ and $g$ with respect to $\overrightarrow{u}$ as 
\begin{align*}
    &\|g(x,t,u_0)\|_{L^{\infty}(0,T;L_p^2(\Omega))}\leq L\|u_0\|_{L^{\infty}(0,T;L_p^2(\Omega))},\\
    &\|f_i(x,t,\overrightarrow{u})\|_{L^{\infty}(0,T;L_p^2(\Omega))}\leq L_i\|\overrightarrow{u}\|_{L^{\infty}(0,T;L_p^2(\Omega))},\, i=0,1,2,3.
\end{align*}
For the derivation, let us take $g(x,t,u_0)$ as an example, which holds 
\begin{align*}
    \|g(x,t,u_0)\|^2_{L^{\infty}(0,T;L_p^2(\Omega))}&=\sup\limits_{t\in [0,T]}\int_{\Omega} |g(x,t,u_0)|^2 \mathrm{d}x\\
    &\leq \sup\limits_{t\in [0,T]} \int_{\Omega}L^2|u_0|^2 \mathrm{d}x\\
    &=L^2 \|u_0\|_{L^{\infty}(0,T;L_p^2(\Omega))},
\end{align*}
where we have made use of the fact $g(x,t,0) = 0$ by equation \eqref{equ:expansion g}.
\end{remark}
\begin{proof}
	In system \eqref{equ:nonlinear system with zeros initial}, we use $u_0$ and $u_i$  as test functions, respectively. After performing integration by parts, we obtain
	\begin{align*}
		&\int_{\Omega}u_0\partial_t u_0+c_1 (\Delta u_0)^2\mathrm{d}x=\int_{\Omega} \left(f_0(x,t,\overrightarrow{u})+g(x,t,u_0)\right)\Delta u_0\mathrm{d}x,\\
		&\int_{\Omega} u_i\partial_t u_i+c_2|\nabla u_i|^2\mathrm{d}x=\int_{\Omega}u_if_i(x,t,\overrightarrow{u})\mathrm{d}x,\, i=1,2,3.
	\end{align*} 
 Summing over these four equations yields
	\begin{align*}
		&\frac{1}{2}\frac{d}{dt}\left(\|u_0\|^2_{L_p^2(\Omega)}+\sum\limits_{i=1}^3\|u_i\|^2_{L_p^2(\Omega)}\right)+c_1\|\Delta u_0\|^2_{L_p^2(\Omega)}+c_2\sum\limits_{i=1}^3\|\nabla u_i\|^2_{L_p^2(\Omega)}\\
		= &\int_{\Omega} \left(f_0(x,t,\overrightarrow{u})+g(x,t,u_0)\right)\Delta u_0\mathrm{d}x+\sum\limits_{i=1}^3\int_{\Omega}u_if_i(x,t,\overrightarrow{u})\mathrm{d}x.
	\end{align*}
	The right-hand side can be bounded by
	\begin{align*}
		\int_{\Omega} &\left(f_0(x,t,\overrightarrow{u})+g(x,t,u_0)\right)\Delta u_0\mathrm{d}x\\
		&\leq \varepsilon_3\|f_0(x,t,\overrightarrow{u})+g(x,t,u_0)\|_{L_p^2(\Omega)}^2+\frac{1}{4\varepsilon_3}\|\Delta u_0\|_{L_p^2(\Omega)}^2\\
		&\leq 2\varepsilon_3\left (\|f_0(x,t,\overrightarrow{u})\|_{L_p^2(\Omega)}^2+\|g(x,t,u_0)\|_{L_p^2(\Omega)}^2\right)+\frac{1}{4\varepsilon_3}\|\Delta u_0\|_{L_p^2(\Omega)}^2\\
		&\leq C_1\varepsilon_3\sum\limits_{j=0}^3\|u_j\|^2_{L_p^2(\Omega)}+\frac{1}{4\varepsilon_3}\|\Delta u_0\|_{L_p^2(\Omega)}^2,
		\end{align*}
		and
		\begin{align*}
		&\int_{\Omega}u_if_i(x,t,\overrightarrow{u})\mathrm{d}x\\
		\leq & \varepsilon_4\|f_i(x,t,\overrightarrow{u})\|_{L_p^2(\Omega)}^2+\frac{1}{4\varepsilon_4} \|u_i\|_{L_p^2(\Omega)}^2\\
		\leq & C_2\varepsilon_4 \sum\limits_{j=0}^3\|u_j\|^2_{L_p^2(\Omega)}+\frac{1}{4\varepsilon_4} \|u_i\|_{L_p^2(\Omega)}^2,\ i=1,2,3,
	\end{align*}
	where $C_1$ and $C_2$ depend on $L$ and $L_i$,  and $\varepsilon_3,\varepsilon_4$ are arbitrary positive number. By setting $\varepsilon_3 = \varepsilon_4=1/(2c_1)$, one can derive that
	\begin{align*}
		&\frac{d}{dt}\left(\|u_0\|^2_{L_p^2(\Omega)}+\sum\limits_{i=1}^3\|u_i\|^2_{L_p^2(\Omega)}\right)+\|\Delta u_0\|^2_{L_p^2(\Omega)}+\sum\limits_{i=1}^3\|\nabla u_i\|^2_{L_p^2(\Omega)}\\
		\leq& C_3 \left(\|u_0\|^2_{L_p^2(\Omega)}+\sum\limits_{i=1}^3\|u_i\|^2_{L_p^2(\Omega)}\right).
	\end{align*}
	Integrating both sides over the interval $[0,s]$, we have
	\begin{align*}
		&\sum\limits_{j=0}^3 \|u_j(\cdot,s)\|^2_{L_p^2(\Omega)}+\|\Delta u_0\|^2_{L^2(0,s;L_p^2(\Omega))}+\sum\limits_{i=1} \|\nabla u_i\|^2_{L^2(0,s;L_p^2(\Omega))}\\\
		\leq &C_3\int_0^s\sum\limits_{j=0}^3\|u_j\|^2_{L_p^2(\Omega)}\mathrm{d}t.
	\end{align*}
	Finally, utilizing the integral form of Gronwall inequality, we conclude
	\begin{align*}
		\sum\limits_{j=0}^3 \|u_j(\cdot,s)\|^2_{L_p^2(\Omega)}=0,
	\end{align*}
	for any $s\in [0,T]$, which completes the proof.
\end{proof}

\end{document}